\documentclass[reqno,10pt]{amsart}
\usepackage{amsmath,amsfonts,amssymb,amsxtra,latexsym,amscd,amsthm,verbatim,bm}
\usepackage[dvipsnames]{xcolor}
\usepackage{graphicx}
\usepackage{enumitem}
\usepackage{hyperref}
\hypersetup{colorlinks=true, pdfstartview=FitV, linkcolor=BrickRed,citecolor=BrickRed, urlcolor=BrickRed}
\definecolor{labelkey}{rgb}{0.6,0,0}

\usepackage[margin=1.30in]{geometry}
\numberwithin{equation}{section}

\providecommand{\abs}[1]{\left\lvert#1\right\rvert}
\providecommand{\norm}[1]{\left\|#1\right\|}

\newcommand{\bmu}{\bm{\mu}}
\newcommand{\bvarrho}{\bm{\varrho}}

\newcommand{\PPsi}{\widetilde{\Psi}}
\newcommand{\RR}{\widetilde{R}}

\def\R{\mathbb{R}}

\newtheorem{theorem}{Theorem}[section]
\newtheorem{lemma}[theorem]{Lemma}

\newtheorem{proposition}[theorem]{Proposition}

\theoremstyle{remark}
\newtheorem{remark}[theorem]{Remark}

\begin{document}

\title[Stability of a point charge for the Vlasov-Poisson system: the radial case]{Stability of a point charge for the Vlasov-Poisson system:\\ the radial case}

\begin{abstract}
We consider the Vlasov-Poisson system with initial data a small, radial, absolutely continuous perturbation of a point charge. We show that the solution is global and disperses to infinity via a modified scattering along trajectories of the linearized flow.

This is done by an exact integration of the linearized equation, followed by the analysis of the perturbed Hamiltonian equation in action-angle coordinates.
\end{abstract}

\author{Benoit Pausader}
\address{Brown University}
\email{benoit\_pausader@brown.edu}

\author{Klaus Widmayer}
\address{\'Ecole Polytechnique F\'ed\'erale de Lausanne}
\email{klaus.widmayer@epfl.ch}

\maketitle

\section{Vlasov-Poisson near a point charge}
This article is devoted to the study of the time evolution and asymptotic behavior of a three dimensional gas of charged particles (a plasma) that interact with a point charge. Under suitable assumptions this system can be described via a measure $M$ on $\mathbb{R}^3_{x}\times\mathbb{R}^3_v$ that is transported by the long-range electrostatic force field created by the charge distribution, resulting in the Vlasov-Poisson system 
\begin{equation}
\begin{split}
\partial_tM+\hbox{div}_{x,v}\left(M\mathfrak{V}\right)=0,\qquad \mathfrak{V}=v\cdot \nabla_x+\nabla_x\phi\cdot\nabla_v,\qquad \Delta_x\phi=\int_{\mathbb{R}^3_v}Mdv.
\end{split}
\end{equation}
Since this equation is rotationally invariant, the Dirac mass $M_{eq}=\delta=\delta_{(0,0)}(x,v)$ is a formal stationary solution and we propose to investigate its stability. We consider initial data\footnote{Here the initial continuous density $f_0=\mu_0^2$ is assumed to be non-negative, a condition which is then propagated by the flow and allows us to work with functions $\mu$ in an $L^2$ framework rather than a general non-negative function $f$ in $L^1$ -- see also our previous work \cite{IPWW2020} for more on this.} of the form $M=q_c\delta+q_g\mu^2_0dxdv$, where $q_c>0$ is the charge of the Dirac mass and $q_g>0$ is the charge per particle of the gas, which results in \emph{purely repulsive interactions}. We track the singular and the absolutely continuous parts of a solution as $M(t)=q_c\delta_{(\bar{x}(t),\bar{v}(t))}+q_g\mu^2(t)dxdv$, which yields the coupled system
\begin{equation}\label{NewVP}
\begin{split}
\left(\partial_t+v\cdot\nabla_x+\frac{q}{2}\frac{x-\bar{x}(t)}{\vert x-\bar{x}(t)\vert^3}\cdot\nabla_v\right)\mu+\lambda\nabla_x\psi\cdot\nabla_v\mu&=0,\qquad\Delta_x\psi=\varrho=\int_{\mathbb{R}^3_v}\mu^2dv,\\
\frac{d\bar{x}}{dt}=\bar{v},\qquad\frac{d\bar{v}}{dt}=\overline{q}\nabla_x\psi(\bar{x})
\end{split}
\end{equation}
where $\lambda=q_g^2/(\epsilon_0m_g)>0$, $q=q_cq_g/(2\pi\epsilon_0 m_g)>0$, $\overline{q}=q_cq_g/(\epsilon_0 m_c)>0$ are positive constants\footnote{In these formulas, $\epsilon_0$ is the vacuum permittivity, $m_g$ the inertia of a gas particle and $m_c$ the inertia of the point charge.}.
\subsection{Main result}
Our main result concerns \eqref{NewVP} with radial initial data,  where the point charge is located at the origin. For sufficiently small initial charge distributions $\mu$ we establish the existence and uniqueness of global, strong solutions and we describe their asymptotic behavior as a modified scattering dynamic. While our full result can be most adequately stated in more adapted ``action-angle'' variables (see Theorem \ref{thm:mainfull} below on page \pageref{thm:mainfull}), for the sake of readability we begin here by giving a (weaker) version in standard Cartesian coordinates:
\begin{theorem}\label{thm:main_rough}
Given any radial initial data $\mu_0\in C^1_c(\mathbb{R}^\ast_+\times\mathbb{R})$, there exists $\varepsilon^\ast>0$ such that for any $0<\varepsilon<\varepsilon^\ast$, there exists a unique global strong solution of \eqref{NewVP} with initial data
\begin{equation*}
(\bar{x}(t=0),\bar{v}(t=0))=(0,0),\qquad \mu(t=0)=\varepsilon\mu_0.
\end{equation*}
Moreover, the electric field decays pointwise and there exists an asymptotic profile $\gamma_\infty\in L^2(\mathbb{R}^\ast_+\times\mathbb{R})$ and a Lagrangian map $(\mathcal{R},\mathcal{V})$ such that
\begin{equation*}
\begin{split}
\mu(\mathcal{R},\mathcal{V},t)\to \gamma_\infty(r,v),\qquad t\to\infty.
\end{split}
\end{equation*}

\end{theorem}

\begin{remark}
\begin{enumerate}
\item Our main theorem is in fact much more precise and requires fewer assumptions, but is better stated in adapted ``action angle'' variables. We refer to Theorem \ref{thm:mainfull}.

\item The Lagrangian map can be written in terms of an asymptotic ``electric field profile'' $\mathcal{E}_\infty$:
\begin{equation*}
\begin{split}
\mathcal{R}(r,v,t)&=t\sqrt{v^2+\frac{q}{r}}-\frac{rq}{2(q+rv^2)}\ln(t)-\lambda\mathcal{E}_\infty(\sqrt{v^2+\frac{q}{r}})\ln(t)+O(1),\\
\mathcal{V}(r,v,t)&=\sqrt{v^2+\frac{q}{r}}-\frac{q}{2\sqrt{v^2+\frac{q}{r}}}\frac{1}{t}+O(\frac{\ln t}{t^2}).
\end{split}
\end{equation*}
The first term corresponds to conservation of the energy along trajectories, the second term comes from a linear correction and the third term on the first line comes from a nonlinear correction to the position. This can be compared with the asymptotic behavior close to vacuum in \cite{IPWW2020,Pan2020} by setting $q=0$.

\end{enumerate}
\end{remark}

\subsubsection{Prior work}
In the absence of a point charge, the Vlasov-Poisson system has been extensively studied and we only refer to \cite{BD1985,LP1991,Pfa1992,Sch1991} for early references on global wellposedness and dispersion analysis, to \cite{CK2016,IPWW2020,Pan2020} for more recent results describing the asymptotic behavior, to \cite{Gla1996,Rei2007} for book references and  to \cite{BM2018} for a historical review.

The presence of a point charge introduces singular electric fields and significantly complicates the analysis. Nevertheless, global existence and uniqueness of strong solutions when the support of the density is separated from the point charge has been established in \cite{MMP2011}, see also \cite{CM2010} and references therein, while global existence of weak solutions for more general support was proved in \cite{DMS2015} with subsequent improvements in \cite{LZ2017,LZ2018,Mio2016}. We also refer to \cite{CLS2018} where ``Lagrangian solutions'' are studied and to \cite{CMMP2012,CZW2015} for works in the case of attractive interactions. Concentration, creation of a point charge and subsequent lack of uniqueness were studied in a related system for ions in $1d$, see \cite{MMZ1994,ZM1994}. To the best of our knowledge, there are no works concerning the asymptotic behavior of such solutions.

The stability of other equilibriums has been considered for the Vlasov-Poisson system with repulsive interactions, most notably in connection to Landau damping \cite{BMM2018,FR2016,HNR2019,MV2011}. In the case of Vlasov-Poisson with attractive interactions, there are many more equilibriums and their linear and nonlinear (in)stability have been studied \cite{GL2017,GS1995,LMR2008,Mou2013,Pen1960}, but the analysis of asymptotic stability is very challenging. We also refer to \cite{IJ2019} which studies the stability of a Dirac mass in the context of the $2d$ Euler equation.

\subsubsection{Our approach}

In previous works on \eqref{NewVP}, the Lagrangian approach allows to integrate the solutions against characteristics but faces the problem of a singular electric field, while a purely Eulerian method leads to a poor control of the solutions, which makes it difficult to study the asymptotic behavior. In this paper, we introduce a different method based on the decomposition of the Hamiltonian to rewrite \eqref{NewVP} as
\begin{equation*}
\begin{split}
\partial_t\mu+\{\mathcal{H}_{0}+\mathcal{H}_{pert},\mu\}=0,
\end{split}
\end{equation*}
where the linearized Hamiltonian $\mathcal{H}_0$ is given in \eqref{LinearHamiltonian} and the nonlinear Hamiltonian $\mathcal{H}_{pert}$ corresponds to the self-generated electrostatic potential \eqref{DefPPsi}. In short, our approach combines a Lagrangian analysis of the linearized problem with an Eulerian PDE framework in the nonlinear analysis, all the while respecting the symplectic structure. This amounts to considering solutions as superpositions of measures on each trajectory of the linearized flow instead of measures on the whole phase space.

On a technical level, one faces the two difficulties of a singular transport field and the nonlinearity separately: the singular electric field created by the point charge is present in the linearized equation coming from $\mathcal{H}_0$, which is integrated exactly. The nonlinearity comes from the perturbed Hamiltonian $\mathcal{H}_{pert}$, but this leads to a simple nonlinear equation, with a nonlinearity which is smoothing.

More precisely, in a first step we analyze the characteristic equations of the linear problem associated to \eqref{NewVP}. These turn out to be the classical ODEs of the Kepler problem, which can be integrated in adapted ``action-angle'' coordinates. In these, the geometry of the characteristic curves is straightened and the linear flow is solved explicitly as a linear map. To treat the nonlinear problem, we conjugate by the linear flow and study the resulting unknown in an Eulerian, $L^2$ based PDE framework, based on energy estimates as in our recent work on the vacuum case \cite{IPWW2020}. This allows us to propagate the required regularity and moments to obtain a global strong solution. Moreover, we can readily identify the asymptotic dynamic: in a mixing type mechanism, the dependence on the ``angles'' is eliminated from the asymptotic electrostatic fields, and the scattering of solutions is modified by a field defined in terms of the ``actions''.

We remark on some features and context of our techniques.
\begin{enumerate}
 \item Since the system \eqref{NewVP} is Hamiltonian and we solve the linearized system through a canonical change of unknown (i.e.\ a diffeomorphism respecting the symplectic structure), the nonlinear problem becomes quite simple after conjugation, see \eqref{eq:VPPoisson}.
 \item The moments we propagate are conserved by the linearized flow, unlike the physical moments in $\langle r\rangle$, $\langle v\rangle$. In fact, even in the nonlinear problem it is quite direct to globally propagate moments in action-angle variables, which already gives the existence of global weak solutions.  
 \item The asymptotic dynamic is easy to exhibit in action-angle variables through inspection of the formulas for the asymptotic electrostatic fields (see \eqref{AsymptoticDynamicsIntro}).
 \item It is notable that we do not require any separation between the point charge and the continuous distribution $\mu$, addressing a question raised in \cite[p.\ 376]{DMS2015}, (see also \cite{MMP2011}).
 \item We expect the methods presented here, based on integration of the linearized equation through ``action-angle'' coordinates, to be broadly applicable, both for local existence of rough solutions and especially for the analysis of long time behavior whenever the linearized equation corresponds to a completely integrable ODE without closed trajectories. This should include a large number of radial problems for plasmas since $1+1$ Hamiltonian ODEs can be integrated by phase portrait.
 \item The usefulness of action-angle variables for the Vlasov-Poisson equation was already exhibited in \cite{FHR2021,GL2017,horsi2017} where the authors produce a large class of $1d$ BGK-type waves which are linearly stable.
\end{enumerate}

\subsubsection{Remarks on the physical setup}
Our primary interest here is to study the interaction of a gas of ions or electrons interacting with a (similarly) charged particle, subject only to electrostatic forces. In this case, up to rescaling, we may assume that $\lambda=1$ in \eqref{NewVP}.

Taking into account gravitational effects, we may also consider the more general case of a large charged and massive point particle with mass $m_c$ and charge $q_c$ interacting with a gas of small particles with mass-per-particle $m_g$ and charge-per-particle $q_g$ subject to both gravitational and electrostatic forces. In this case, our result holds whenever the principal gas-point particle interaction is repulsive, i.e. when (in appropriates physical units)
\begin{equation}\label{CondOpenTraj}
q_cq_g> m_cm_g,
\end{equation}
whereas (due to the small data assumption on the gas at initial time) the gas-gas interaction may be repulsive $(\lambda=1)$ or attractive $(\lambda=-1$) depending on the sign of $(q_g)^2-(m_g)^2$. 

The situation that is beyond the scope of our analysis is the case when the inequality in \eqref{CondOpenTraj} is reversed and some trajectories of the linearized system are closed. Note that in this case, even the local existence theory is incomplete.

\subsection{Overview and ideas of the proof}

We note that in the particular case of radial initial conditions\footnote{This is a strong notion of radial solutions. A weaker notion would be to consider functions which are jointly invariant under rotations, i.e.\ a density of the form $\mu=\mu(\vert x\vert,\vert v\vert,\ell)$, where $\ell=x\times v$ is the microscopic angular momentum. We refer e.g. to \cite{Pan2020} for the study of such solutions.}
\begin{equation}
 \begin{cases}
  \mu(x,v,t=0)&=\mu_0(\abs{x},\abs{v}),\\ (\bar{x},\bar{v})(t=0)&=(0,0),
 \end{cases}
\end{equation}
the Dirac mass in \eqref{NewVP} does not move (i.e.\ $\bar{x}(t)=\bar{v}(t)=0$) and the continuous particle distribution $\mu(t)$ of the solution is a radial function. The equations \eqref{NewVP} reduce to the following system for $\mu(x,v,t)$:
\begin{equation}\label{eq:NewVPrad}
 \left(\partial_t+v\cdot\nabla_x+\frac{q}{2}\frac{x}{\vert x\vert^3}\cdot\nabla_v\right)\mu+\lambda\nabla_x\psi\cdot\nabla_v\mu=0,\qquad\Delta_x\psi=\varrho=\int_{\mathbb{R}^3_v}\mu^2dv.
\end{equation}
Per a slight abuse of notation with $r:=\abs{x}$, $\varrho(r,t)=\varrho(x,t)$, by radiality the electric field $E:=-\nabla_x\psi$ of the ensemble can be computed as
\begin{equation}
 E(x,t)=-\nabla_x\psi(x,t)=-\partial_r\psi(r,t)\frac{x}{r},\qquad \partial_r\psi(r,t)=\frac{1}{r^2}\int_{s=0}^r\varrho(s,t)s^2ds.
\end{equation}

\subsubsection{The ``radial'' phase space}
Since as discussed the equations \eqref{NewVP} are invariant under spherical symmetry and we will work with spherically symmetric data, it is more convenient to work on the phase space $(r,v)\in\R_{+}^\ast\times\R$ (rather than $(x,v)\in\R^3\times\R^3$). Note that $r^2v^2drdv$ is the natural measure corresponding to that of radially symmetric functions on $\R^3\times\R^3$, and hence we will work with the new density
\begin{equation}
 \bm\mu(r,v,t):=rv\mu(r,v,t).
\end{equation}
This is chosen such that the (conserved) mass is the square of the $L^2$ norm of both $\bm\mu$ on $\R_+^\ast\times\R$ and $\mu$ on $\R^3\times\R^3$, i.e.\ we have
\begin{equation}
 \iint_{r,v}\bm\mu^2(r,v,t)\, drdv=\iint_{x,v}\mu^2(x,v,t)\, dxdv=\iint_{x,v}\mu^2_0(x,v)\,dxdv.
\end{equation}
Moreover, the equations for $\bm\mu$ simply read
\begin{equation}\label{eq:VPrad}
\begin{aligned}
 &\left(\partial_t+v\partial_r+\frac{q}{2r^2}\partial_v\right)\bm\mu=\lambda{\bm E}\partial_v\bmu,\\&\qquad{\bm E}(r,t):=-\partial_r\psi(r,t)=\frac{1}{r^2}\int_{s=0}^r\bvarrho(s,t)\,ds,\quad \bvarrho(s,t)=\int\bmu^2(s,v,t)\,dv.
\end{aligned} 
\end{equation}
This equation is Hamiltonian and can be equivalently written as
\begin{equation}\label{NonlinearHamiltonianStructure}
\begin{split}
2\partial_t\bm\mu+\{\mathcal{H},\bm\mu \}=0,\qquad  \mathcal{H}(r,v):=v^2+\frac{q}{r}-2\lambda\psi(r),\qquad \{f,g\}&:=\partial_vf\partial_rg-\partial_rf\partial_vg,
\end{split}
\end{equation}
which leads to the conservation of energy
\begin{equation*}
\begin{split}
{\bf \mathcal{H}}_{total}(t)&:=\iint \mathcal{H}(r,v)\cdot \bm\mu^2\,drdv=\iint \left(v^2+\frac{q}{ r}\right)\cdot \bm\mu^2\,drdv+\lambda\int {\bm E}^2\cdot r^2dr={\bf \mathcal{H}}_{total}(0).
\end{split}
\end{equation*}

\subsubsection{The linearized system}
In order to study \eqref{eq:VPrad}, we first consider the linearized equation for a function $f(r,v,t)$:
\begin{equation}\label{eq:linVP_intro}
\begin{split}
\left(\partial_t+v\partial_r+\frac{q}{2r^2}\partial_v\right)f=0.
\end{split}
\end{equation}
This linear transport equation can be solved directly via its characteristic equations
\begin{equation}\label{eq:ODE_intro}
\dot{r}=v,\qquad\dot{v}=\frac{q}{2r^2}.
\end{equation}
One recognizes here the classical Kepler problem in the radial setting, which can be integrated using generalized\footnote{Action-angle variables traditionally refer to the case when the trajectories are bounded.} ``action angle'' coordinates $(\theta, a)$ (see also Figure \ref{fig:action-angle}).

\begin{figure}[h]
 \centering
 \includegraphics[width=0.9\textwidth]{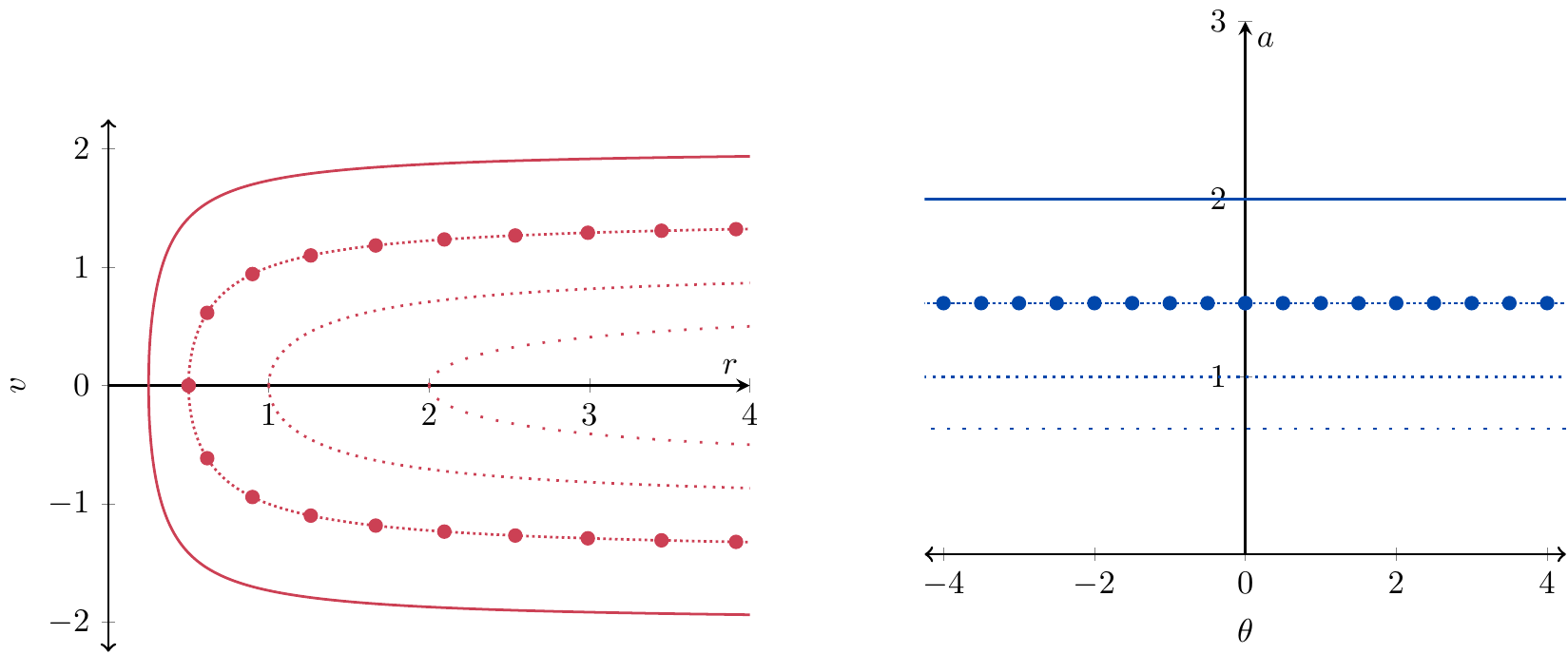}
 \caption{The trajectories of solutions of \eqref{eq:ODE_intro}, on the left in physical space $(r,v)$ coordinates, on the right in action-angle coordinates $(\theta,a)$ from Lemma \ref{LinearizedProblemLem}. The large dots correspond to points that are uniformly spaced in $\theta$ for a fixed $a$ (right), and their image in $(r,v)$ (left).}
 \label{fig:action-angle}
\end{figure}

{\begin{lemma}\label{LinearizedProblemLem}
 There exists a canonical transformation to ``action-angle'' coordinates:
 \begin{equation}
\mathbb{R}_+^\ast\times\mathbb{R}\ni (r,v)\mapsto(\Theta(r,v),\mathcal{A}(r,v))\in\mathbb{R}\times\mathbb{R}_+^\ast
 \end{equation}
 with inverse $(R(\theta,a),V(\theta,a))$, such that $f$ solves \eqref{eq:linVP_intro} if and only if
 \begin{equation*}
 g(\theta,a,t):=f(R(\theta,a),V(\theta,a),t)
 \end{equation*}
 solves the free streaming equation
 \begin{equation}\label{FS}
 \left(\partial_t+a\partial_\theta\right)g=0.
 \end{equation}
 \end{lemma}}

\subsubsection{Nonlinear analysis}
Since \eqref{FS} can be solved directly, we conjugate by this change of variables to stabilize the linearized system, thus defining $\gamma$ as follows:
\begin{equation}\label{NewUnknown}
\begin{split}
\gamma(\theta,a,t)&:=\bmu(R(\theta+ta,a),V(\theta+ta,a),t),\\
\bmu(r,v,t)&=\gamma(\Theta(r,v)-t\mathcal{A}(r,v),\mathcal{A}(r,v),t).
\end{split}
\end{equation}
Since the change of variable preserves the symplectic structure, we find that the full nonlinear problem \eqref{eq:VPrad} is equivalent to
\begin{equation}\label{eq:VPPoisson}
  \partial_t\gamma=\lambda\{\PPsi,\gamma\},\qquad\{f,g\}=\partial_af\partial_\theta g-\partial_\theta f\partial_ag.
\end{equation}
Here the potential can be expressed in action angle coordinates as follows:
\begin{equation}\label{DefPPsi}
 \PPsi(\theta,a,t):=\iint_{\vartheta,\alpha}\frac{1}{\max\{R(\theta+ta,a),R(\vartheta+t\alpha,\alpha)\}}\gamma^2(\vartheta,\alpha,t)\, d\vartheta d\alpha.
\end{equation}

\begin{remark}
 While the trajectories of the linear equation \eqref{eq:linVP} are straight lines in action-angle variables, in physical variables they correspond to an incoming ray followed by an outgoing one traced at varying velocities. 
 
 For the nonlinear problem this creates extra challenges, as interactions can occur over vastly disparate spatial scales. As the below Figure \ref{fig:caustics} illustrates, in some regimes the evolution $R(\theta+ta,a)$ is not a simple function of $(\theta,a)$, from which one of the two variables can be recovered once the other is known (see also Lemma \ref{LemDiracMeasureDensity} below).
 
\begin{figure}[h]
 \centering
 \includegraphics[width=0.7\textwidth]{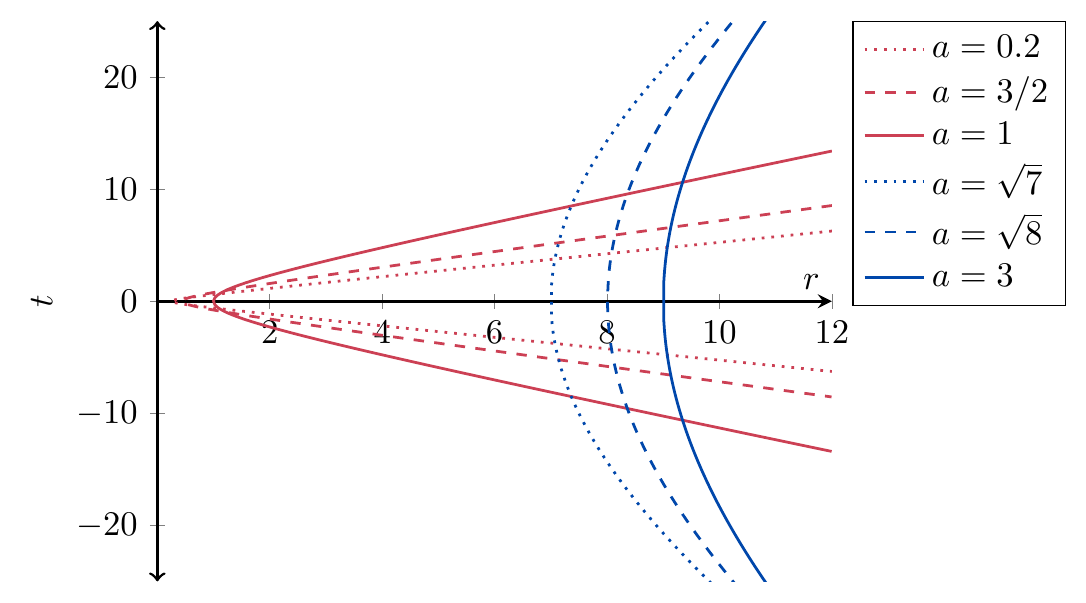}
 \caption{Fronts $R(ta,a)$ of particles in the nonlinear problem, starting at $\theta=0$.}
 \label{fig:caustics}
\end{figure}

\end{remark}

It remains to study solutions to \eqref{eq:VPPoisson}. The dispersion mechanism is accounted for through the conjugation with the linearized flow and we hope to show that this picture remains true when we add the remaining nonlinear contribution, i.e.\ we expect that solutions to \eqref{eq:VPPoisson} do not change too much over time. We first use a bootstrap argument to propagate strong norms, which suffices to obtain global existence and decay of the electric field:

\begin{proposition}\label{prop:deriv_boot}
There exists $\varepsilon^\ast$ such that for all $0<\varepsilon_0\le \varepsilon_1\le\delta<\varepsilon^\ast$, the following holds. Let $\gamma$ be a solution to \eqref{eq:VPPoisson} with initial data $\gamma_0$ on $0\le t\le T$ and assume that for $0\le t\le T$,
\begin{equation*}
\begin{split}
\Vert \left(a^{-20}+\theta^{20}+a^{20}\right)\gamma_0\Vert_{L^2_{\theta,a}}+\Vert (a+a^{-1})\partial_\theta\gamma_0\Vert_{L^2_{\theta,a}}+\Vert a\partial_a\gamma_0\Vert_{L^2_{\theta,a}}&\le\varepsilon_0,\\
\Vert \left(a^{-20}+\theta^{20}+a^{20}\right)\gamma(t)\Vert_{L^2_{\theta,a}}+\Vert (a+a^{-1})\partial_\theta\gamma(t)\Vert_{L^2_{\theta,a}}+\Vert a\partial_a\gamma(t)\Vert_{L^2_{\theta,a}}&\le \varepsilon_1\langle t\rangle^{\delta},\\
\end{split}
\end{equation*}
then in fact
\begin{equation*}
\begin{split}
\Vert \left(a^{-20}+a^{20}\right)\gamma\Vert_{L^2_{\theta,a}}+\Vert (a+a^{-1})\partial_\theta\gamma\Vert_{L^2_{\theta,a}}&\le\varepsilon_0+\varepsilon_1^\frac{3}{2},\\
\Vert \theta^{20}\gamma\Vert_{L^2_{\theta,a}}+\Vert a\partial_a\gamma\Vert_{L^2_{\theta,a}}&\le\varepsilon_0+\varepsilon_1^\frac{3}{2} t^{\delta}.
\end{split}
\end{equation*}
\end{proposition}

This in turn allows us to investigate the asymptotic behavior. It can easily be formally deduced once one observes that, given the bounds propagated by the bootstrap, on the support of the density, one has
\begin{equation*}
\begin{split}
\vert R(\theta+ta,a)-ta\vert=o(t),
\end{split}
\end{equation*}
so that one expects that $\PPsi$ is asymptotically independent of $\theta$:
\begin{equation*}
\begin{split}
\PPsi(\theta,a,t)&\simeq\frac{1}{t}\int_\alpha \frac{1}{\max\{a,\alpha\}}\mathcal{Z}(\alpha)\, d\alpha=\frac{1}{t}{ \Phi}(a),\qquad\mathcal{Z}(\alpha):=\lim_{t\to\infty}\int\gamma^2(\theta,\alpha,t)\, d\theta.
\end{split}
\end{equation*}
As a consequence, \eqref{eq:VPPoisson} becomes a perturbation of a shear equation:
\begin{equation}\label{AsymptoticDynamicsIntro}
\partial_t\gamma=\frac{\lambda}{t}\partial_a{ \Phi}\cdot \partial_\theta\gamma.
\end{equation}
which can easily be integrated. This can be made rigorous under appropriate assumptions on the initial data and it leads to our main result.
\begin{theorem}\label{thm:mainfull}

There exists $\varepsilon_0>0$ such that any initial data $\gamma_0$ satisfying
\begin{equation}\label{BdID}
\Vert (\theta^{20}+a^{20}+a^{-20})(\gamma_0,\partial_\theta\gamma_0,\partial_a\gamma_0)\Vert_{L^2_{\theta,a}}\le\varepsilon\le\varepsilon_0
\end{equation}
leads to a unique global solution $\gamma(t)\in C^1_tL^2_{\theta,a}\cap C^0_tH^1_{\theta,a}$ of \eqref{eq:VPPoisson} with an electric field $(\partial_\theta\PPsi,\partial_a\PPsi)$ which decays in time. In addition, there exists $\gamma_\infty\in L^2_{\theta,a}\cap C^0_\theta L^2_a$ such that
\begin{equation}\label{L2Convergence}
\begin{split}
\Vert \gamma(\theta+\lambda\ln t\cdot \mathcal{E}_\infty(a),a,t)-\gamma_\infty(\theta,a)\Vert_{L^2_{\theta,a}}\lesssim\varepsilon t^{-\frac{1}{10}},
\end{split}
\end{equation}
where
\begin{equation}\label{SecondDefEF}
\mathcal{E}_\infty(a):=\frac{1}{a^2}\iint\mathfrak{1}_{\{\alpha\le a\}}\gamma^2_\infty(\vartheta,\alpha)\, d\vartheta d\alpha.
\end{equation}

\end{theorem}
We expect that the number of moments in \eqref{BdID} can be significantly reduced. It is interesting that the decay of the electric field can be obtained under much weaker assumptions (see Lemma \ref{MomBdsLem}), but it is unclear to us how much asymptotic information can be recovered in this case.

\subsection{Organization of the paper}

In Section \ref{SecLin}, we study the ODE associated to the linearized equation and establish a number of geometric results and bounds on relevant quantities for the nonlinear problem. In Section \ref{SecNLin}, we study the nonlinear equation; we establish the moment bootstrap for weak solutions in Section \ref{SSecBootstap}, the derivative bootstrap for strong solutions in Section \ref{BootDerSSSec} and prove the modified scattering in Section \ref{ModifiedScat} by first obtaining a weak-strong limit for scattering data in Section \ref{ConvScatteringQties} and finally the convergence of the particle density in Section \ref{ConvPDF}.

We emphasize that one can obtain decay of moments for weak solutions in a self-contained way using only the results of the Sections \ref{FlowMapSSec}, \ref{EFEPSSec} and \ref{MomentSSSec}.

\section{Linearized equation}\label{SecLin}

The goal of this section is to integrate the linearized problem \eqref{eq:linVP} and to prove various estimates for the corresponding transfer functions. Lemma \ref{LinearizedProblemLem} follows easily from Lemma \ref{ExplicitActionAngleLem} below.

\subsection{Straightening the linear characteristics}\label{FlowMapSSec}

The linearization of \eqref{NonlinearHamiltonianStructure} at $\bm\mu=0$ is the Hamiltonian differential equation associated to
\begin{equation}\label{LinearHamiltonian}
\mathcal{H}_0(r,v):=v^2+\frac{q}{r},
\end{equation}
namely
\begin{equation}\label{eq:linVP}
\begin{split}
\left(\partial_t+v\partial_r+\frac{q}{2r^2}\partial_v\right)\mu=0.
\end{split}
\end{equation}
This is now a linear transport equation, which can be integrated easily once we know the trajectories of the corresponding ODE:
\begin{equation}\label{ODE}
\dot{r}=v,\qquad\dot{v}=\frac{q}{2r^2}.
\end{equation}

\subsubsection{Radial trajectories}

Since we consider a $1+1$ Hamiltonian system \eqref{ODE}, the trajectories can be integrated by phase portrait. Letting $\mathcal{A}=\sqrt{\mathcal{H}_0}$, we can explicitly integrate the resulting equation
\begin{equation*}
\dot{r}=\sqrt{\mathcal{A}^2-\frac{q}{r}}
\end{equation*}
by starting the ``clock'' $\theta=0$ at the periapsis (i.e.\ the point of closest approach): Let
\begin{equation}\label{eq:var_def}
\begin{split}
\mathcal{A}(r,v)&=\sqrt{v^2+\frac{q}{r}},\qquad\Theta(r,v)=\frac{v}{\vert v\vert}r_{min}G(\frac{r}{r_{min}}),\qquad r_{min}(r,v)=\frac{q}{v^2+\frac{q}{r}},\\
R(\theta,a)&=r_{min}H(\frac{\vert \theta\vert}{r_{min}}),\qquad V(\theta,a)=\frac{\theta}{\vert\theta\vert}\sqrt{a^2-\frac{q}{R}},\qquad r_{min}(a)=\frac{q}{a^2},
\end{split}
\end{equation}
where $G:(1,\infty)\to\mathbb{R}$ and $H:\mathbb{R}_+\to(1,\infty)$ satisfy
\begin{equation}\label{eq:GH_def}
\begin{split}
G(1)&=0,\quad G^\prime(s)=\left[1-\frac{1}{s}\right]^{-\frac{1}{2}},\qquad G(s)=\sqrt{s(s-1)}+\ln\left(\sqrt{s}+\sqrt{s-1}\right),\\
H(x)&=G^{-1}(x),\quad H^\prime(x)=\sqrt{\frac{H(x)-1}{H(x)}}.
\end{split}
\end{equation}
These functions and related ones are studied in more details in Section \ref{StructureFunctionsSSSec}. We can now solve the linear problem \eqref{eq:linVP} via a canonical change of variable:
\begin{lemma}\label{ExplicitActionAngleLem}
 The change of variables
 \begin{equation}
 (r,v)\mapsto(\Theta(r,v),\mathcal{A}(r,v))
 \end{equation}
 in \eqref{eq:var_def} defines a canonical diffeomorphism of the phase space $(r,v)$ (with inverse $(R(\theta,a),V(\theta,a))$ as in \eqref{eq:var_def}), which linearizes the flow in the sense that for the flow map $\Phi^t(r,v)$ associated to the Hamiltonian ODEs \eqref{ODE} we have
 \begin{equation}\label{IntegralOfLinearMotion}
 \Theta(\Phi^t(r,v))-\Theta(r,v)=t\mathcal{A}(r,v),\qquad\mathcal{A}(\Phi^t(r,v))=\mathcal{A}(r,v).
 \end{equation}
 Moreover, we have
 \begin{equation*}
  \det \frac{\partial(\Theta,\mathcal{A})}{\partial(r,v)}=\det\frac{\partial(R,V)}{\partial(\theta,a)}=1.
 \end{equation*}
 \end{lemma}

\begin{proof}
That $(\Theta,\mathcal{A})$ and $(R,V)$ are inverse can be checked directly once one observes that $r_{min}$ is consistent: $r_{min}(\mathcal{A}(r,v))=r_{min}(r,v)$ and $r_{min}(a)=r_{min}(R(a,\theta),V(a,\theta))$. It is direct to check that $\mathcal{A}$ is conserved along the flow. Moreover,
\begin{equation*}
\begin{aligned}
\frac{d}{dt}\Theta(r(t),v(t))&=\vert v\vert\left(1-\frac{r_{min}}{r}\right)^{-\frac{1}{2}}+\frac{q}{2r^2}\delta(v)r_{min}G(\frac{r}{r_{min}})\\
&\qquad +\frac{v}{\vert v\vert}\left[v\partial_rr_{min}+\frac{q}{2r^2}\partial_vr_{min}\right](G-\frac{r}{r_{min}}G^\prime)(\frac{r}{r_{min}}).
\end{aligned}
\end{equation*}
The first term on the right hand side gives $\mathcal{A}$, and the second vanishes since when $v=0$, $G(r/r_{min})=G(1)=0$, while direct computations show that the bracket in the last term vanishes. In addition, the same computations show that
\begin{equation*}
\begin{split}
d\Theta\wedge d\mathcal{A}&=\frac{q}{\mathcal{A}r^2}r_{min}G(\frac{r}{r_{min}})\delta(v)dr\wedge dv+\frac{v}{\vert v\vert\mathcal{A}}(G-\frac{r}{r_{min}}G^\prime)(\frac{r}{r_{min}})\cdot \left(v\partial_rr_{min}+\frac{q}{2r^2}\partial_vr_{min}\right)dr\wedge dv\\
&\quad+\frac{\vert v\vert}{\mathcal{A}}G^\prime(\frac{r}{r_{min}})dr\wedge dv\\
&=dr\wedge dv,
\end{split}
\end{equation*}
which shows that the transformation preserves the symplectic form and hence has Jacobian $1$.
\end{proof}

\subsubsection{Study of the structure functions}\label{StructureFunctionsSSSec}

In this subsection, we study the geometric functions that arise from the change of variable in Lemma \ref{ExplicitActionAngleLem}. These are independent of assumptions on the solutions.

\begin{figure}[h]
 \centering
 \includegraphics[width=0.6\textwidth]{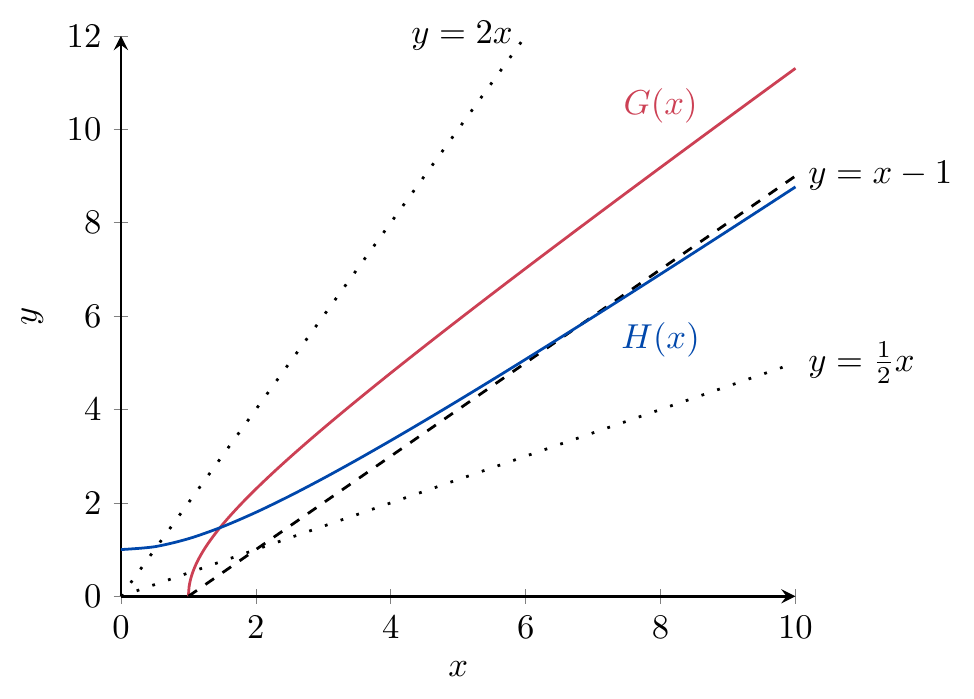}
 \caption{Plots of the functions $G$ and $H$ of Lemma \ref{BehaviorGH}.}
 \label{fig:GH}
\end{figure}

\begin{lemma}\label{BehaviorGH}
The functions $G$ and $H$ are almost linear
\begin{equation}\label{GHAlmostlinear}
\begin{split}
x-1\le& G(x)\le x+\ln\left(2\sqrt{x}\right)\le 2x,\qquad 1\le x<\infty,\\
\frac{x}{2}\le& H(x)\le x+1,\qquad 0\le x<\infty.
\end{split}
\end{equation}
In addition, we note the asymptotic behavior of $G$ and its inverse
 \begin{equation}\label{EstimG}
 \begin{split}
 G(s)&=s+\frac{1}{2}\ln s+\ln 2-\frac{1}{2}+O_{s\to\infty}(\frac{1}{s}),\qquad G(1+\hbar)=2\sqrt{\hbar}+O_{\hbar\to0}(\hbar^{3/2}),\\
 H(x)&=x-\frac{1}{2}\ln(x)-\ln 2+\frac{1}{2}+\frac{1}{4}\frac{\ln(x)}{x}+O_{x\to\infty}(\frac{1}{x}),\qquad H(h)=1+\left(\frac{h}{2}\right)^2+O_{h\to0}\left(\frac{h}{2}\right)^5.
 \end{split}
 \end{equation} 
\end{lemma}

\begin{proof}[Proof of Lemma \ref{BehaviorGH}]
Since $G$ can be integrated explicitly, we easily obtain the first line in \eqref{GHAlmostlinear} from \eqref{eq:GH_def}, and the second line follows from the fact that $H=G^{-1}$. Now \eqref{EstimG} follows by expanding the expression for $G$.
\end{proof}

In addition, we will frequently consider first and second order derivatives.
\begin{lemma}\label{ControlLLemma}
We have explicit formulas for the first order derivatives
\begin{equation*}
\begin{split}
\partial_\theta R(\theta,a)&=\frac{\theta}{\vert\theta\vert}H^\prime(\frac{a^2\vert\theta\vert}{q}),\qquad 1-\frac{q}{a^2R(\theta,a)}\le\frac{\theta}{\vert\theta\vert}\partial_\theta R(\theta,a)\le 1,\\
\partial_aR(\theta,a)&=\frac{q}{a^3}\left(\int_{s=0}^{\frac{a^2\vert\theta\vert}{q}}\frac{s}{H^2(s)}ds-2\right),\qquad\vert\partial_aR\vert\lesssim \frac{q\ln\langle\frac{a^2}{q}R(\theta,a)\rangle}{a^3},
\end{split}
\end{equation*}
and the formulas for the second order derivatives
\begin{equation*}
\begin{split}
\partial_\theta\partial_\theta R(\theta,a)&=\frac{q}{2a^2R^2(\theta,a)},\qquad
\partial_\theta\partial_aR(\theta,a)=\frac{1}{a}\frac{\theta}{\vert\theta\vert}\frac{\frac{a^2\vert\theta\vert}{q}}{H^2(\frac{a^2\vert\theta\vert}{q})},\\
\partial_a\partial_aR(\theta,a)&=\frac{q}{a^4}\left(6-3\int_{s=0}^{\frac{a^2\vert\theta\vert}{q}}\frac{s}{H^2(s)}ds+\frac{2\left(\frac{a^2\vert\theta\vert}{q}\right)^2}{H^2(\frac{a^2\vert\theta\vert}{q})}\right),
\end{split}
\end{equation*}
so that
\begin{equation*}
\begin{split}
\vert\partial_\theta\partial_\theta R(\theta,a)\vert&\le\frac{q}{2a^2R^2(\theta,a)},\qquad \vert\partial_\theta\partial_aR(\theta,a)\vert\le\frac{1}{a}\frac{\frac{a^2\vert\theta\vert}{q}}{H^2(\frac{a^2\vert\theta\vert}{q})},\qquad
\vert\partial_a\partial_aR(\theta,a)\vert\le \frac{q}{a^4}\ln\langle \frac{a^2}{q}R(\theta,a)\rangle.
\end{split}
\end{equation*}

\end{lemma}

\begin{proof}[Proof of Lemma \ref{ControlLLemma}]

The formulas follow from direct calculations using that
\begin{equation*}\label{DiffEqH}
H(0)=1,\qquad 1\ge H^\prime(x)=\left[1-1/H\right]^{\frac{1}{2}}\ge1-H^{-1},\qquad H^{\prime\prime}(x)=1/(2H^2(x)),
\end{equation*}
together with Lemma \ref{BehaviorGH} to control the behavior of $H$.

\end{proof}

\subsubsection{Kinematics of linear trajectories}

Here we collect a few estimates on the behavior of the trajectories of \eqref{ODE}. Using \eqref{IntegralOfLinearMotion}, we see that the linearized flow is simple in the action-angle variables. For simplicity of notation, given a function $F(\theta,a)$ in phase space, we will denote by
\begin{equation}\label{LinEvolFunc}
\widetilde{F}(\theta,a):=F(\theta+ta,a)
\end{equation}
its evolution under the linear flow. This is a slight abuse of notation since the transformation depends on time; however all our estimates will be instantaneous so this should not lead to confusion. Since we will show that in action-angle variables, the new density is (almost) stable, we expect that the main role will be played by trajectories starting from the ``bulk region'' $\mathcal{B}$ defined by
\begin{equation}\label{DefBulkZone}
\begin{split}
\mathcal{B}&:=\{a\ge t^{-\frac{1}{4}},\,\, \vert\theta\vert\le ta/2\},\qquad\mathcal{B}^c=\{a\le t^{-\frac{1}{4}}\,\,\hbox{ or }\,\,\vert\theta\vert\ge ta/2\}.
\end{split}
\end{equation}

We start with a few simple observations. By definition, we have a universal lower bound for $R$,
\begin{equation}\label{EstimRa}
\begin{split}
a^2R(\theta,a)\ge q,
\end{split}
\end{equation}
but in the bulk region, one can be more precise.

\begin{lemma}\label{DerRLem}

We have the following control on $\widetilde{R}$:
\begin{equation}\label{BdRTFirstDer}
\begin{split}
\vert\partial_\theta\widetilde{R}(\theta,a)\vert\le 1,\qquad\vert\partial_a\widetilde{R}(\theta,a)\vert\lesssim t+\frac{q}{a^3}\ln\langle\frac{a^2}{q}\widetilde{R}(\theta,a)\rangle
\end{split}
\end{equation}
and
\begin{equation}\label{BdRTSecondDer}
\begin{split}
\vert\partial_\theta\partial_\theta\widetilde{R}(\theta,a)\vert&\lesssim \widetilde{R}^{-1}(\theta,a),\qquad
\vert\partial_\theta\partial_a\widetilde{R}(\theta,a)\vert\lesssim \frac{t+a^{-3}}{\widetilde{R}(\theta,a)},\\
\vert\partial_a\partial_a\widetilde{R}(\theta,a)\vert&\lesssim \frac{t^2}{a^2\widetilde{R}^2(\theta,a)}+\frac{t}{a^3\widetilde{R}(\theta,a)}+\frac{q}{a^4}\ln\langle \frac{a^2}{q}\widetilde{R}(\theta,a)\rangle.
\end{split}
\end{equation}

In addition, we have a more precise control in the bulk: when $(\theta,a)\in\mathcal{B}$, we have that
\begin{equation*}
\begin{split}
ta/2\le \widetilde{R}(\theta,a)\le 2ta,\qquad 1-\frac{2q}{ta^3}\le\partial_\theta\widetilde{R}(\theta,a)\le 1,\qquad  \partial_a\widetilde{R}(\theta,a)\ge\frac{3}{4}t,\qquad\vert\partial_a\partial_a\widetilde{R}\vert\lesssim a^{-4}\ln\langle ta^3\rangle,
\end{split}
\end{equation*}
and in particular, the change of variable $a\mapsto \widetilde{R}(\theta,a)$ is well behaved.

\end{lemma}

\begin{proof}[Proof of Lemma \ref{DerRLem}]

The bounds \eqref{BdRTFirstDer} and \eqref{BdRTSecondDer} follow from Lemma \ref{ControlLLemma} and the formulas
\begin{equation*}
\begin{split}
\partial_\theta\widetilde{R}(\theta,a)=\partial_\theta R(\theta+ta,a),\qquad \partial_a\widetilde{R}(\theta,a)=\left[t\partial_\theta R+\partial_aR\right](\theta+ta,a),\\
\partial_\theta\partial_\theta\widetilde{R}(\theta,a)=\partial_\theta\partial_\theta R(\theta+ta,a),\qquad\partial_\theta\partial_a\widetilde{R}=\left(t\partial_\theta\partial_\theta R+\partial_a\partial_\theta R\right)(\theta+ta,a),\\
\partial_a\partial_a\widetilde{R}(\theta,a)=\left(t^2\partial_\theta\partial_\theta R+2t\partial_a\partial_\theta R+\partial_a\partial_aR\right)(\theta+ta,a).
\end{split}
\end{equation*}
Now, in the bulk, we observe that $ta/2\le \vert\theta+ta\vert\le 2ta$, and $ta^3\gg1$ so that by \eqref{GHAlmostlinear} we have $ta/2\le \widetilde{R}(\theta+ta,a)\le 2ta$. The other bounds follow directly.

\end{proof}

Since the interaction involves quantities defined in the physical space, it will be useful to understand how to relate them to phase space variables. The next lemma is concerned with solutions of the equation
\begin{equation}\label{InvertR}
\begin{split}
\widetilde{R}(\theta,a)=r
\end{split}
\end{equation}
for fixed $t$ and $r$.
\begin{figure}[h]
 \centering
 \includegraphics[width=0.7\textwidth]{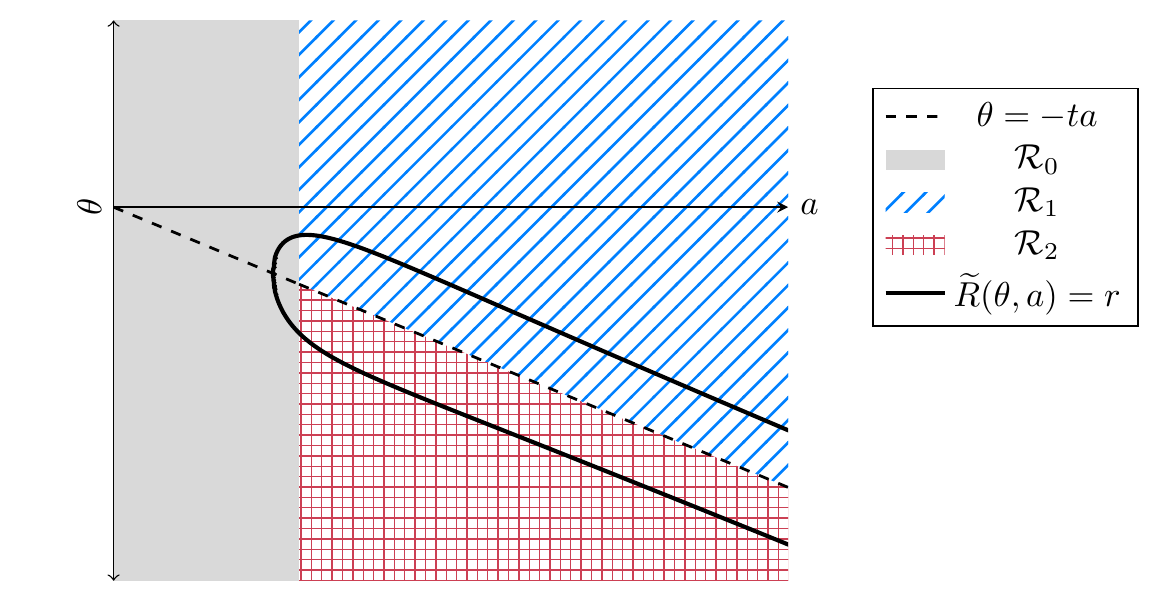}
 \caption{The different regions of Lemma \ref{LemDiracMeasureDensity}.}
 \label{fig:areas}
\end{figure}

\begin{lemma}\label{LemDiracMeasureDensity}

Let
\begin{equation*}
\begin{split}
A:=\sqrt{\frac{q}{r}}(1+\hbar),\qquad \hbar:=c\cdot \min\{1,\frac{r^3}{q^5t^2}\}.
\end{split}
\end{equation*}
for some fixed small constant $c>0$ and define the regions
\begin{equation*}
\begin{split}
\mathcal{R}_0:=\{0\le a\le A,\,\, \theta\in\mathbb{R}\},\qquad\mathcal{R}_1:=\{a\ge A,\,\, \theta\le -ta\},\qquad\mathcal{R}_2:=\{a\ge A,\,\,\theta\ge-ta\}.
\end{split}
\end{equation*}
In the region $\mathcal{R}_0$, we see that for any $\theta$, there exists at most one $a=\aleph(\theta;r,t)$ solution of \eqref{InvertR}. In addition, we have that
\begin{equation}\label{PropertiesR0}
\begin{split}
\sqrt{\frac{q}{r}}\le \aleph\le 2\sqrt{\frac{q}{r}},\qquad \vert\theta+t\aleph\vert\le 3\hbar r,\qquad \vert\partial_a\widetilde{R}(\theta,\aleph)\vert\gtrsim \frac{r^\frac{3}{2}}{q^\frac{5}{2}}\gtrsim \frac{a^{-3}}{q}.
\end{split}
\end{equation}

In the region $\mathcal{R}_j$, $j\in\{1,2\}$, for each choice of $a$, there exists exactly one $\theta=\tau_j(a;r,t)$ solution of \eqref{InvertR}. In addition, we see that
\begin{equation}\label{CharacR1}
\begin{split}
\tau_1&\le -ta-r/2,\qquad a^2r\ge q,\qquad\partial_\theta\widetilde{R}(\tau_1,a)\gtrsim \hbar,\\
\end{split}
\end{equation}
while on $\mathcal{R}_2$, we see that
\begin{equation}\label{CharacR2}
\begin{split}
\partial_\theta\widetilde{R}(\tau_1,a)\gtrsim \hbar\qquad\hbox{ and }\qquad\hbox{either }1\le\frac{a^2r}{q}\le C,\quad\hbox{ or }\quad\partial_a\widetilde{R}(\tau_2,a)\ge ct.
\end{split}
\end{equation}

\end{lemma}

\begin{proof}

We can rewrite \eqref{InvertR} as
\begin{equation*}
\begin{split}
\frac{q}{a^2}H(\frac{a^2\vert\theta+ta\vert}{q})=r\quad\Leftrightarrow\quad \vert\theta+ta\vert=\frac{q}{a^2}G(\frac{a^2 r}{q})\quad\Leftrightarrow\quad \theta=-ta\pm \frac{q}{a^2}G(\frac{a^2 r}{q}).
\end{split}
\end{equation*}
We start with $\mathcal{R}_0$ and denote $a^\ast=\sqrt{q/r}$, $\vartheta=\theta+ta$ and $x=a^2\vert\vartheta\vert/q$ so that we are considering the equation
\begin{equation}\label{SolInvertR}
\begin{split}
r=\frac{q}{a^\ast}=\frac{q}{a^2}H(\frac{a^2\vert\vartheta\vert}{q})\quad\Leftrightarrow\quad x=G((\frac{a}{a^\ast})^2).
\end{split}
\end{equation}
Now let $a=a^\ast\sqrt{1+h^2}$ for some $h\ge0$. We have that by \eqref{EstimG} and Lemma \ref{ControlLLemma}
\begin{equation}\label{NewEqAAA}
\begin{split}
E_1=x-G((\frac{a}{a^\ast})^2)&=x-2h+O(h^3),\\
\partial_a\widetilde{R}(\theta,a)&=t\frac{\vartheta}{\vert\vartheta\vert}H^\prime(x)-\frac{2}{q (a^\ast)^3}\left[1+h^2\right]^{-\frac{3}{2}}\left(1-\int_{s=0}^x\frac{s}{2H^2(s)}ds\right).
\end{split}
\end{equation}
From this we see that if $h$ is small enough, $0\le h\le c$, there exists a unique solution $x$ to $E_1=0$, and this solution satisfies $h\le x\le 3h$. In addition, if $h\le c/(q(a^\ast)^3t)$, there holds that $\partial_a\widetilde{R}\lesssim -1/(q(a^\ast)^3)$.

Now in region $\mathcal{R}_j$, $j\in\{1,2\}$, we see that
\begin{equation*}
\begin{split}
x=G(1+h^2)\ge h\ge\hbar
\end{split}
\end{equation*}
and in particular, using \eqref{SolInvertR}, we find that
\begin{equation*}
\begin{split}
\tau_1(a;r,t)&:=-ta- \frac{q}{a^2}G(\frac{a^2 r}{q})<-ta<
\tau_2(a;r,t):=-ta+\frac{q}{a^2}G(\frac{a^2 r}{q}).
\end{split}
\end{equation*}
In addition, we have that
\begin{equation*}
\begin{split}
\partial_\theta \widetilde{R}(\theta,a)&=H^\prime(x)\gtrsim\hbar.\end{split}
\end{equation*}
and the other bounds in \eqref{CharacR1} follow directly from the definitions. 

Finally, the last statement in \eqref{CharacR2} follows from the formula for $\partial_a\widetilde{R}$ in \eqref{NewEqAAA}: Note that $\vartheta/\abs{\vartheta}=1$ and letting $x=b$ the bound where the term in parenthesis vanishes, then when $0\le x\le 2b$, there holds that $q\le a^2r\le H(2b)$, while for $x\ge 2b$ we see that both terms have same sign and $H^{\prime\prime}\ge0$, so that
\begin{equation*}
\begin{split}
\partial_a\widetilde{R}(\theta,a)\ge tH^\prime(2b).
\end{split}
\end{equation*}

\end{proof}

\subsection{Electric field and potential}\label{EFEPSSec}
Given an instantaneous density distribution $\mu(r,v,t)$, it is useful to introduce the ``physical potential'' ${\bf\Psi}$ of the associated electric field as in \eqref{eq:VPrad}, explicitly given as
\begin{equation}
 {\bf\Psi}(r,t)=\iint_{r,s}\frac{1}{\max\{r,s\}}\mu^2(s,v,t)ds dv=\iint_{\vartheta,\alpha}\frac{1}{\max\{r,R(\vartheta+t\alpha,\alpha)\}}\gamma^2(\vartheta,\alpha,t)\, d\vartheta d\alpha,
\end{equation}
when $\gamma(\vartheta,\alpha,t)$ is the corresponding density distribution in action-angle coordinates as in \eqref{NewUnknown}. Then \eqref{DefPPsi} can be rewritten as $\PPsi(\theta,a,t)={\bf \Psi}(\widetilde{R}(\theta,a),t)$. This allows to obtain formulas for the derivatives of $\PPsi$ in action angle variables in terms of the electric field $\bf E$, the local mass ${\bf m}$ and the density $\bm\varrho$ (compare also \eqref{eq:VPrad}):
\begin{equation}\label{DefENewVar}
\begin{split}
{\bf E}(r,t)=-\partial_r{\bf \Psi}(r,t)&=\frac{{\bf m}(r,t)}{r^2},\qquad {\bf m}(r,t):=\iint_{\vartheta,\alpha}\mathfrak{1}_{\{R(\vartheta+t\alpha,\alpha)\le r\}}\gamma^2(\vartheta,\alpha,t)\, d\vartheta d\alpha,\\
\bm\varrho(r,t)&:=\partial_r{\bf m}(r,t)=\iint\delta(R(\vartheta+t\alpha,\alpha)-r)\cdot\gamma^2(\vartheta,\alpha,t)d\vartheta d\alpha.
\end{split}
\end{equation}
Then for $\beta\in\{a,\theta\}$ we have
\begin{equation}\label{DerPPsi}
\begin{split}
\partial_\beta\PPsi&=-\frac{{\bf m}(\widetilde{R}(\theta,a),t)}{\widetilde{R}^2(\theta,a)}\cdot\partial_\beta \widetilde{R}(\theta,a),\\
\partial_\alpha\partial_\beta\PPsi&=-\frac{{\bf m}(\widetilde{R})}{\widetilde{R}^2}\left(\partial_\alpha\partial_\beta\widetilde{R}-2\frac{\partial_\alpha\widetilde{R}\partial_\beta\widetilde{R}}{\widetilde{R}}\right)-\bm\varrho(\widetilde{R})\cdot\frac{\partial_\alpha\widetilde{R}\partial_\beta\widetilde{R}}{\widetilde{R}^2}.
\end{split}
\end{equation}

We note that the local mass has a trivial uniform bound
\begin{equation}\label{UniformBdLocalMass1}
0\le {\bf m}(r)\le{\bf m}(\infty):=\Vert\gamma\Vert_{L^2_{\theta,a}}^2
\end{equation}
but this can be made more precise.
\begin{lemma}\label{lem:ControlOnEField}
We can decompose ${\bf m}$ as
\begin{equation*}
\begin{split}
{\bf m}(r,t)&={\bf m}_s(r,t)+{\bf m}_n(r,t),
\end{split}
\end{equation*}
where we have that for any $\ell,\kappa>0$
\begin{equation}\label{DecMass}
\begin{split}
0\le {\bf m}_s(r,t)&\lesssim \left(\frac{r}{t}\right)^\ell\Vert a^{-\frac{\ell}{2}}\gamma\Vert_{L^2_{\theta,a}}^2,\\
0\le {\bf m}_n(r,t)&\lesssim \left(\frac{r}{t}\right)^\ell t^{-\frac{\kappa-\ell}{2}}\left[\Vert a^{-\kappa}\gamma\Vert_{L^2_{\theta,a}}^2+\Vert a^{\frac{\ell-\kappa}{2}}\theta^{\frac{\ell+\kappa}{2}}\gamma\Vert_{L^2_{\theta,a}}^2\right].
\end{split}
\end{equation}

\end{lemma}

\begin{proof}[Proof of Lemma \ref{lem:ControlOnEField}]

The decomposition corresponds to localizing in and out of the bulk zone defined in \eqref{DefBulkZone}. Thus
\begin{equation*}
\begin{split}
{\bf m}_s(r,t)&:= \iint \mathfrak{1}_{\{\widetilde{R}(\vartheta,\alpha)\le r\}}\mathfrak{1}_{\mathcal{B}}\cdot\gamma^2(\vartheta,\alpha)\,d\vartheta d\alpha,\qquad {\bf m}_n(r,t):=\iint \mathfrak{1}_{\{\widetilde{R}(\vartheta,\alpha)\le r\}}\mathfrak{1}_{\mathcal{B}^c}\cdot\gamma^2(\vartheta,\alpha)\,d\vartheta d\alpha.
\end{split}
\end{equation*}
Using Lemma \ref{DerRLem}, we see that
\begin{equation*}
\begin{split}
{\bf m}_s(r,t)&\le r^\kappa\iint \mathfrak{1}_{\mathcal{B}}\cdot\gamma^2(\vartheta,\alpha)\,\frac{d\vartheta d\alpha}{\widetilde{R}^\kappa(\vartheta,\alpha)}\lesssim (r/t)^{\kappa}  \iint \alpha^{-\kappa}\cdot\gamma^2(\vartheta,\alpha)\,d\vartheta d\alpha,
\end{split}
\end{equation*}
while using \eqref{EstimRa},
\begin{equation*}
\begin{split}
\frac{{\bf m}_n(r,t)}{r^\kappa}&\le\iint \mathfrak{1}_{\mathcal{B}^c}\cdot\gamma^2(\vartheta,\alpha)\,\frac{d\vartheta d\alpha}{\widetilde{R}^\kappa(\vartheta,\alpha)}\lesssim \iint \left[\mathfrak{1}_{\{\vert\alpha\vert\le t^{-\frac{1}{4}}\}}+\mathfrak{1}_{\{\vert\vartheta\vert\ge t\alpha/2\}}\right]\cdot \alpha^{2\kappa}\cdot\gamma^2(\vartheta,\alpha)\,d\vartheta d\alpha.
\end{split}
\end{equation*}

\end{proof}

We will use the following consequences:
\begin{proposition}\label{BoundsOnE}

There holds that
\begin{equation}\label{BdsDThetaPsi}
\begin{split}
t^\frac{3}{2}\cdot \sup_{\theta,a}\frac{1}{a} \vert\partial_\theta\PPsi(\theta,a)\vert &\lesssim \Vert a^{-\frac{3}{4}}\gamma\Vert_{L^2_{\theta,a}}^2+t^{-\frac{1}{4}}\Vert (a^{-2}+\theta^2)\gamma\Vert_{L^2_{\theta,a}}^2
\end{split}
\end{equation}
and
\begin{equation}\label{BdsDaPsi}
\begin{split}
t\cdot \vert\partial_a\PPsi(\vartheta,\alpha)\vert &\lesssim \Vert a^{-1}\gamma\Vert_{L^2_{\theta,a}}^2+\Vert\gamma\Vert_{L^2_{\theta,a}}^2\cdot \left[\vert\vartheta\vert+\alpha^{-3}\right]+t^{-\frac{1}{4}}\Vert (a^{-\frac{5}{2}}+\theta+\theta^{\frac{5}{2}})\gamma\Vert_{L^2_{\theta,a}}^2.
\end{split}
\end{equation}

\end{proposition}

\begin{proof}[Proof of Proposition \ref{BoundsOnE}]
 Using \eqref{EstimRa} and \eqref{BdRTFirstDer}, we find that
\begin{equation*}
\begin{split}
\frac{1}{a} \vert\partial_\theta\PPsi(\theta,a)\vert&=\frac{{\bf m}(\widetilde{R}(\theta,a),t)}{a\widetilde{R}^2(\theta,a)}\vert\partial_\theta\widetilde{R}(\theta,a)\vert\le \frac{1}{\sqrt{q}}\frac{{\bf m}(\widetilde{R}(\theta,a),t)}{\widetilde{R}^\frac{3}{2}(\theta,a)}
\end{split}
\end{equation*}
and we can use \eqref{DecMass} with $\ell=3/2$, $\kappa=2$. For \eqref{BdsDaPsi}, we use \eqref{BdRTFirstDer} to get that
\begin{equation*}
\begin{split}
\vert\partial_a\PPsi(\theta,a)\vert&=\frac{{\bf m}(\widetilde{R}(\theta,a),t)}{\widetilde{R}^2(\theta,a)}\vert\partial_a\widetilde{R}(\theta,a)\vert\lesssim t\frac{{\bf m}(\widetilde{R}(\theta,a),t)}{\widetilde{R}^2(\theta,a)}+\frac{{\bf m}(\widetilde{R}(\theta,a),t)}{\widetilde{R}^\frac{1}{2}(\theta,a)}\frac{\ln\langle\frac{a^2}{q}\widetilde{R}(\theta,a)\rangle}{\sqrt{q}(\frac{a^2}{q}\widetilde{R}(\theta,a))^\frac{3}{2}}.
\end{split}
\end{equation*}
From \eqref{DecMass} with $\ell=2$, $\kappa=\frac{5}{2}$ we obtain
\begin{equation*}
\begin{split}
t\frac{{\bf m}(\widetilde{R}(\theta,a),t)}{\widetilde{R}^2(\theta,a)}\lesssim t^{-1}\Vert a^{-1}\gamma\Vert_{L^2_{\theta,a}}^2+t^{-\frac{5}{4}}\Vert (a^{-\frac{5}{2}}+\theta^{\frac{5}{2}})\gamma\Vert_{L^2_{\theta,a}}^2.
\end{split}
\end{equation*}
Similarly, if $\frac{a^2}{q}\widetilde{R}(\theta,a)\ge t^\frac{1}{2}$, we can use \eqref{DecMass} ($\ell=\frac{1}{2}$, $\kappa=1$) to get
\begin{equation*}
\begin{split}
\frac{{\bf m}(\widetilde{R}(\theta,a),t)}{\widetilde{R}^\frac{1}{2}(\theta,a)}\frac{\ln\langle\frac{a^2}{q}\widetilde{R}(\theta,a)\rangle}{\sqrt{q}(\frac{a^2}{q}\widetilde{R}(\theta,a))^\frac{3}{2}}\lesssim t^{-\frac{2}{3}}\frac{{\bf m}(\widetilde{R}(\theta,a),t)}{\widetilde{R}^\frac{1}{2}(\theta,a)}\lesssim t^{-\frac{7}{6}}\Vert a^{-1}\gamma\Vert_{L^2_{\theta,a}} +t^{-\frac{7}{6}-\frac{1}{4}}\Vert (a^{-1}+\theta)\gamma\Vert_{L^2_{\theta,a}}.
\end{split}
\end{equation*}
On the other hand, it follows from Lemma \ref{DerRLem} that if $q\le a^2\widetilde{R}(\theta,a)\le t^\frac{1}{2}$, then $(\theta,a)\in\mathcal{B}^c$, and we use that
\begin{equation*}
\begin{split}
\frac{{\bf m}(\widetilde{R}(\theta,a),t)}{\widetilde{R}^\frac{1}{2}(\theta,a)}\frac{\ln\langle\frac{a^2}{q}\widetilde{R}(\theta,a)\rangle}{\sqrt{q}(\frac{a^2}{q}\widetilde{R}(\theta,a))^\frac{3}{2}}\mathfrak{1}_{\mathcal{B}^c}\lesssim a\Vert\gamma\Vert_{L^2_{\theta,a}}^2\mathfrak{1}_{\mathcal{B}^c}\lesssim t^{-1}\Vert\gamma\Vert_{L^2_{\theta,a}}^2\cdot \left[\vert\theta\vert+a^{-3}\right],
\end{split}
\end{equation*}
which gives \eqref{BdsDaPsi}.

\end{proof}

\subsubsection{Study of the density}

Controlling derivatives of $\gamma$ requires estimates on the density; these are obtained in a similar way to the mass (see Lemma \ref{lem:ControlOnEField}), but are more involved. 

\begin{lemma}\label{LemDensity}

The density can be decomposed into two terms,
\begin{equation*}
\begin{split}
\bm\varrho&=\bm\varrho_s+\bm\varrho_n
\end{split}
\end{equation*}
where for $\kappa,\sigma\ge0$,
\begin{equation}\label{DecRho}
\begin{split}
\vert\bm\varrho_s(r,t)/r^\kappa\vert&\lesssim t^{-1-\kappa}\left[\Vert a\partial_a\gamma\Vert_{L^2_{\theta,a}}\Vert a^{-1-\kappa}\gamma\Vert_{L^2_{\theta,a}}+\Vert a^{-\frac{\kappa+1}{2}}\gamma\Vert_{L^2_{\theta,a}}^2\right],\\
\vert\bm\varrho_n(r,t)/r^\kappa\vert& \lesssim t^{-1-\kappa-\sigma}\left[\Vert\partial_\theta\gamma\Vert_{L^2_{\theta,a}}+\Vert a\partial_a\gamma\Vert_{L^2_{\theta,a}}\right]\cdot \Vert (1+a^{-2\kappa-4\sigma-6}+a^{\kappa-\sigma+3}\theta^{\kappa+\sigma+3})\gamma\Vert_{L^2_{\theta,a}}.
\end{split}
\end{equation}

\end{lemma}

The key observation is that the estimate for $\bm\varrho_s$ only involves at most one copy of the large term $\Vert a\partial_a\gamma\Vert_{L^2_{\theta,a}}$.

\begin{proof}[Proof of Lemma \ref{LemDensity}]
We can decompose into two regions
\begin{equation*}
\begin{split}
{\bf m}(r,t)&={\bf m}^1(r,t)+{\bf m}^2(r,t),\qquad\chi^{[1]}(\vartheta,\alpha,t)=\varphi_{\ge1}(t^{-1}\partial_a\widetilde{R}(\vartheta,\alpha)),\qquad\chi^{[2]}=1-\chi^{[1]},\\
{\bf m}^j(r,t)&:=\iint \mathfrak{1}_{\{\widetilde{R}(\vartheta,\alpha)\le r\}}\chi^{[j]}(\vartheta,\alpha,t)\gamma^2(\vartheta,\alpha)d\vartheta d\alpha,
\end{split}
\end{equation*}
where $\varphi_{\ge1}(x)$ denotes a smooth function supported on $\{x\ge1/10\}$ and equal to $1$ for $x\ge 1/2$.

{\bf Study of $\partial_r{\bf m}^1$}. This contains the main term. Integrating by parts, we observe that
\begin{equation*}
\begin{split}
\partial_r{\bf m}^1(r,t)&=\iint \partial_r\left(\mathfrak{1}_{\{\widetilde{R}(\vartheta,\alpha)\le r\}}\right)\chi^{[1]}(\vartheta,\alpha,t)\gamma^2(\vartheta,\alpha)d\vartheta d\alpha\\
&=\iint -\frac{1}{\partial_\alpha\widetilde{R}(\vartheta,\alpha)}\partial_\alpha\left(\mathfrak{1}_{\{\widetilde{R}(\vartheta,\alpha)\le r\}}\right)\chi^{[1]}(\vartheta,\alpha,t)\gamma^2(\vartheta,\alpha)d\vartheta d\alpha\\
&=\iint \mathfrak{1}_{\{\widetilde{R}(\vartheta,\alpha)\le r\}} \partial_\alpha\left(\frac{\chi^{[1]}(\vartheta,\alpha,t)}{\partial_\alpha\widetilde{R}(\vartheta,\alpha)}\gamma^2(\vartheta,\alpha)\right)d\vartheta d\alpha= \bm\varrho^1_s(r,t)+\bm\varrho^2_s(r,t)+M^{1,1}+M^{1,2},\\
\end{split}
\end{equation*}
where
\begin{equation*}
\begin{split}
 \bm\varrho_s^1(r,t)&:=\iint \mathfrak{1}_{\mathcal{B}}\mathfrak{1}_{\{\widetilde{R}(\vartheta,\alpha)\le r\}} \partial_\alpha\left(\frac{\chi^{[1]}(\vartheta,\alpha,t)}{\partial_\alpha\widetilde{R}(\vartheta,\alpha)}\right)\gamma^2(\vartheta,\alpha)d\vartheta d\alpha,\\
M^{1,1}&:=\iint \mathfrak{1}_{\mathcal{B}^c}\mathfrak{1}_{\{\widetilde{R}(\vartheta,\alpha)\le r\}} \partial_\alpha\left(\frac{\chi^{[1]}(\vartheta,\alpha,t)}{\partial_\alpha\widetilde{R}(\vartheta,\alpha)}\right)\gamma^2(\vartheta,\alpha)d\vartheta d\alpha,\\
 \bm\varrho_s^2(r,t)&:=2\iint \mathfrak{1}_{\mathcal{B}}\mathfrak{1}_{\{\widetilde{R}(\vartheta,\alpha)\le r\}}\frac{\chi^{[1]}(\vartheta,\alpha,t)}{\partial_\alpha\widetilde{R}(\vartheta,\alpha)}\gamma(\vartheta,\alpha)\cdot\partial_\alpha\gamma(\vartheta,\alpha)d\vartheta d\alpha\\
M^{1,2}&:=2\iint \mathfrak{1}_{\mathcal{B}^c}\mathfrak{1}_{\{\widetilde{R}(\vartheta,\alpha)\le r\}}\frac{\chi^{[1]}(\vartheta,\alpha,t)}{\partial_\alpha\widetilde{R}(\vartheta,\alpha)}\gamma(\vartheta,\alpha)\cdot\partial_\alpha\gamma(\vartheta,\alpha)d\vartheta d\alpha.\\
\end{split}
\end{equation*}
From Lemma \ref{DerRLem} we recall that in the bulk region $\RR\sim at$, and thus
\begin{equation*}
\begin{split}
r^{-\kappa}\vert \bm\varrho^2_s(r,t)\vert&\lesssim \iint \mathfrak{1}_{\mathcal{B}}\frac{1}{\widetilde{R}^\kappa(\vartheta,\alpha)}\frac{\chi^{[1]}(\vartheta,\alpha,t)}{\partial_\alpha\widetilde{R}(\vartheta,\alpha)}\vert \gamma(\vartheta,\alpha)\cdot\partial_\alpha\gamma(\vartheta,\alpha)\vert d\vartheta d\alpha\lesssim t^{-\kappa-1}\Vert a\partial_a\gamma\Vert_{L^2_{\theta,a}}\Vert a^{-1-\kappa}\gamma\Vert_{L^2_{\theta,a}},
\end{split}
\end{equation*}
while on the other hand, using \eqref{EstimRa} and \eqref{DefBulkZone},
\begin{equation*}
\begin{split}
r^{-\kappa}\vert M^{1,2}\vert&\lesssim  \iint \mathfrak{1}_{\mathcal{B}^c}\alpha^{2\kappa}\frac{\chi^{[1]}(\vartheta,\alpha,t)}{\partial_\alpha\widetilde{R}(\vartheta,\alpha)}\vert \gamma(\vartheta,\alpha)\cdot\partial_\alpha\gamma(\vartheta,\alpha)\vert d\vartheta d\alpha\\
&\lesssim t^{-1-\kappa-\sigma}\Vert a\partial_a\gamma\Vert_{L^2_{\theta,a}}\left(\Vert a^{-2\kappa-1-4\sigma}\gamma\Vert_{L^2_{\theta,a}}+\Vert a^{\kappa-1}\theta^{\kappa+\sigma}\gamma\Vert_{L^2_{\theta,a}}\right).
\end{split}
\end{equation*}
Direct computations using Lemma \ref{DerRLem} show that
\begin{equation*}
\begin{split}
\left\vert \partial_\alpha\left(\frac{\chi^{[1]}(\vartheta,\alpha,t)}{\partial_\alpha\widetilde{R}(\vartheta,\alpha)}\right)\right\vert&\lesssim \left\vert \frac{\partial_a\partial_a\widetilde{R}(\vartheta,\alpha,t)}{(\partial_a\widetilde{R})^2(\vartheta,\alpha,t)}\right\vert \underline{\chi}^{[1]}(\vartheta,\alpha,t)\lesssim \frac{1}{\widetilde{R}(\vartheta,\alpha)}+\frac{1}{t\alpha }+\frac{\vert\vartheta\vert}{\alpha^2t^2}.
\end{split}
\end{equation*}
Separating the contribution of the bulk and outside as in the proof of Lemma \ref{lem:ControlOnEField}, we find that
\begin{equation*}
\begin{split}
r^{-\kappa}\vert \bm\varrho_s^{1}\vert&\lesssim \iint \mathfrak{1}_{\{\widetilde{R}(\vartheta,\alpha)\le r\}}\mathfrak{1}_{\mathcal{B}} \left(\frac{1}{\widetilde{R}(\vartheta,\alpha)}+\frac{1}{t\alpha }+\frac{\vert\vartheta\vert}{\alpha^2t^2}\right)\gamma^2(\vartheta,\alpha)\frac{d\vartheta d\alpha}{\widetilde{R}^\kappa(\vartheta,\alpha)}\lesssim t^{-1-\kappa}\Vert\alpha^{-\frac{1+\kappa}{2}}\gamma\Vert_{L^2_{\theta,a}}^2,
\end{split}
\end{equation*}
while using \eqref{EstimRa} and \eqref{DefBulkZone} yields
\begin{equation*}
\begin{split}
r^{-\kappa}\vert M^{1,1}\vert&\lesssim \iint \mathfrak{1}_{\{\widetilde{R}(\vartheta,\alpha)\le r\}}\mathfrak{1}_{\mathcal{B}^c} \left(\frac{1}{\widetilde{R}(\vartheta,\alpha)}+\frac{1}{t\alpha }+\frac{\vert\vartheta\vert}{\alpha^2t^2}\right)\gamma^2(\vartheta,\alpha)\frac{d\vartheta d\alpha}{\widetilde{R}^\kappa(\vartheta,\alpha)}\\
&\lesssim t^{-1-\kappa-\sigma}\norm{(a^{-\kappa-1-2\sigma}(1+a)+a^{\frac{1}{2}(\kappa-1-\sigma)}\abs{\theta}^{\frac{1}{2}(\kappa+\sigma)}(1+a\abs{\theta}^{\frac{1}{2}})\gamma}_{L^2_{\theta,a}}^2.
\end{split}
\end{equation*}

\medskip

{\bf Study of $\partial_r{\bf m}^2$}. We now consider
\begin{equation*}
\begin{split}
\partial_r{\bf m}^2(r,t)&=\iint \delta(\widetilde{R}(\vartheta,\alpha)-r)\chi^{[2]}(\vartheta,\alpha,t)\gamma^2(\vartheta,\alpha)d\vartheta d\alpha.\\
\end{split}
\end{equation*}
The main observation is that thanks to Lemma \ref{DerRLem}, we have that
\begin{equation}\label{M2OnlyNoise}
\chi^{[2]}\mathfrak{1}_{\mathcal{B}}=0.
\end{equation}
The Dirac measure restricts to the set studied in Lemma \ref{LemDiracMeasureDensity} and we decompose accordingly
\begin{equation*}
\begin{split}
\partial_r{\bf m}^2(r,t)&=M^{2,0}+M^{2,1}+M^{2,2},\\
M^{2,j}&:=\iint \delta(\widetilde{R}(\vartheta,\alpha)-r)\chi^{[2]}(\vartheta,\alpha,t)\mathfrak{1}_{\mathcal{R}_j}\gamma^2(\vartheta,\alpha)d\vartheta d\alpha.
\end{split}
\end{equation*}
and using \eqref{PropertiesR0}, we see that
\begin{equation*}
\begin{split}
0\le r^{-\kappa}M^{2,0}&= r^{-\kappa}\int \chi^{[2]}\mathfrak{1}_{\mathcal{R}_0}\gamma^2(\vartheta,\aleph)\frac{d\vartheta}{\partial_a\widetilde{R}(\vartheta,\aleph)}\lesssim \int \mathfrak{1}_{\mathcal{B}^c}\gamma^2(\vartheta,\aleph)\aleph^{3+2\kappa}d\vartheta.
\end{split}
\end{equation*}
Integrating $\gamma^2(\vartheta,\aleph)\aleph^{3+2\kappa}=\int_0^{\aleph}\partial_\alpha(\gamma^2(\vartheta,\alpha)\alpha^{3+2\kappa}) d\alpha$ from $0\le\alpha\le\aleph$ (note that $0\le \alpha\le a$ and $a\in\mathcal{B}^c$ implies that $\alpha\in\mathcal{B}^c$), we can estimate
\begin{equation*}
\begin{split}
0\le t^{\kappa+\sigma+1} r^{-\kappa}M^{2,0}&\lesssim t^{\kappa+\sigma+1}\iint \mathfrak{1}_{\mathcal{B}^c}\alpha^{2\kappa}\left(\alpha^2\gamma^2(\vartheta,\alpha)+\vert\gamma(\vartheta,\alpha)\cdot \alpha^3\partial_a\gamma(\vartheta,\alpha)\vert\right)\, d\vartheta d\alpha\\
&\lesssim \Vert a^{-\kappa-1-2\sigma}\gamma\Vert_{L^2_{\theta,a}}^2+\Vert a^{\frac{\kappa+1-\sigma}{2}}\theta^{\frac{\sigma+\kappa}{2}}\gamma\Vert_{L^2_{\theta,a}}^2\\
&\quad+\Vert a\partial_a\gamma\Vert_{L^2_{\theta,a}}\cdot\left(\Vert a^{-2\kappa-2-4\sigma}\gamma\Vert_{L^2_{\theta,a}}+\Vert a^{\kappa+1+\sigma}\theta^{\sigma}\gamma\Vert_{L^2_{\theta,a}}\right).
\end{split}
\end{equation*}
Using \eqref{CharacR1} and integrating over $\vartheta\geq \tau_1$, we see that
\begin{equation*}
\begin{split}
0\le M^{2,1}&= \int \chi^{[2]}\mathfrak{1}_{\mathcal{R}_1}\gamma^2(\tau_{1},\alpha)\frac{d\alpha}{\partial_\theta\widetilde{R}(\tau_{1},\alpha)}\\
&\lesssim (1+r^{-3}t^2)\int \mathfrak{1}_{\{\alpha^2r\ge q,\,\,\vert\tau_{1}\vert\gtrsim \alpha t+r \}}\gamma^2(\tau_{1},\alpha)d\alpha\\
&\lesssim (1+r^{-3}t^2)\iint \mathfrak{1}_{\{\alpha^2r\ge q,\,\,\vert\vartheta\vert\gtrsim \alpha t+r \}}\vert\gamma(\vartheta,\alpha)\partial_\theta\gamma(\vartheta,\alpha)\vert d\vartheta d\alpha\\
&\lesssim t^{-1-\kappa-\sigma}r^\kappa \Vert \partial_\theta\gamma\Vert_{L^2_{\theta,a}}\Vert (1+a^4\theta^2)\theta^{\kappa+\sigma+1} a^{\kappa-1-\sigma}\gamma\Vert_{L^2_{\theta,a}}.
\end{split}
\end{equation*}
Similarly, using \eqref{CharacR2} we obtain
\begin{equation*}
\begin{split}
0\le M^{2,2}&= \iint \chi^{[2]}\mathfrak{1}_{\mathcal{R}_2}\gamma^2(\tau_2,\alpha)\frac{d\alpha}{\partial_\theta\widetilde{R}(\tau_2,\alpha)}\\
&\lesssim (1+r^{-3}t^2)\int\mathfrak{1}_{\{\alpha^2 r\ge q\}}\mathfrak{1}_{\mathcal{B}^c}\gamma^2(\tau_2,\alpha)d\alpha\\
&\lesssim (1+r^{-3}t^2)\iint\mathfrak{1}_{\{\alpha^2 r\ge q\}}\mathfrak{1}_{\{\alpha\le t^{-\frac{1}{4}}\,\,\hbox{ or }\vert\vartheta\vert\ge t\alpha/2\}}\vert\gamma(\vartheta,\alpha)\partial_\theta\gamma(\vartheta,\alpha)\vert d\alpha\\
&\lesssim t^{-1-\kappa-\sigma}r^\kappa \Vert \partial_\theta\gamma\Vert_{L^2_{\theta,a}}\left[\norm{a^{-2\kappa-6-4\sigma}(1+a^5)\gamma}_{L^2_{\theta,a}}+\norm{a^{\kappa-1}\theta^{\kappa+\sigma}(1+\alpha^{4+\sigma})\gamma}_{L^2_{\theta,a}}\right].
\end{split}
\end{equation*}
This finishes the proof with $\bm\varrho_s=\bm\varrho_s^1+\bm\varrho_s^2$ and $\bm\varrho_n=M^{1,1}+M^{1,2}+M^{2,0}+M^{2,1}+M^{2,2}$.
\end{proof}

\begin{remark}\label{rem:aposweights}
As can be seen from the proof of \ref{LemDensity}, we only need positive moments in $a$ to control the area outside of the bulk where $\abs{\theta}\geq a t$. Such moments in $a$ could be replaced by moments in $\theta$, and thus positive weights in $a$ are not necessary for our result.
\end{remark}

\begin{proposition}\label{ControlSecondDerPsi}

There holds that
\begin{equation}\label{ControlSecondDerPsiEst1}
\begin{split}
t^\frac{3}{2}\vert\partial_\theta\partial_a\PPsi\vert+t^2(1+a^{-2})\vert\partial_\theta\partial_\theta\PPsi\vert&\lesssim N_1, \\
\end{split}
\end{equation}
and
\begin{equation}\label{ControlSecondDerPsiEst2}
\begin{split}
\frac{ta^2}{1+a^2}\vert\partial_a\partial_a\PPsi\vert&\lesssim \Vert a\partial_a\gamma\Vert_{L^2_{\theta,a}}\Vert a^{-1}\gamma\Vert_{L^2_{\theta,a}}+\Vert (a^2+a^{-2})\gamma\Vert_{L^2_{\theta,a}}^2+t^{-\frac{1}{3}}N_1
\end{split}
\end{equation}
where
\begin{equation}\label{DefN1}
\begin{split}
N_1&:=\Vert (a^{-20}+a^{20}+\theta^{20})\gamma\Vert_{L^2_{\theta,a}}^2+\Vert a\partial_a\gamma\Vert_{L^2_{\theta,a}}^2+\Vert\partial_\theta\gamma\Vert_{L^2_{\theta,a}}^2.
\end{split}
\end{equation}

\end{proposition}

\begin{proof}[Proof of Proposition \ref{ControlSecondDerPsi}]

The most important term is the term with mixed derivative (see Section \ref{EstimGowthNormDer}). We recall from \eqref{DerPPsi} that
\begin{equation*}
\begin{split}
\partial_\theta\partial_a\PPsi=-\frac{{\bf m}(\widetilde{R})}{\widetilde{R}^2}\left(\partial_\theta\partial_a\widetilde{R}-2\frac{\partial_\theta\widetilde{R}\partial_a\widetilde{R}}{\widetilde{R}}\right)-\bm\varrho(\widetilde{R})\frac{\partial_\theta\widetilde{R}\partial_a\widetilde{R}}{\widetilde{R}^2}.
\end{split}
\end{equation*}
On the one hand, using Lemma \ref{DerRLem}, we see that
\begin{equation*}
\begin{split}
\left\vert\frac{{\bf m}(\widetilde{R})}{\widetilde{R}^2}\left(\partial_\theta\partial_a\widetilde{R}-2\frac{\partial_\theta\widetilde{R}\partial_a\widetilde{R}}{\widetilde{R}}\right)\right\vert&\lesssim t\frac{{\bf m}(\widetilde{R},t)}{\widetilde{R}^3}+\frac{1}{(a^2\widetilde{R})^\frac{3}{2}}\frac{{\bf m}(\widetilde{R},t)}{\widetilde{R}^\frac{3}{2}}+\frac{1}{(a^2\widetilde{R})^\frac{1}{2}}\frac{{\bf m}(\widetilde{R},t)}{\widetilde{R}^\frac{3}{2}},\\
\left\vert\frac{\partial_\theta\widetilde{R}(\theta,a)\partial_a\widetilde{R}(\theta,a)}{\widetilde{R}^2(\theta,a)}\right\vert&\lesssim \frac{1}{(a^2\widetilde{R}(\theta,a))^\frac{1}{2}}\frac{1}{\widetilde{R}^\frac{1}{2}(\theta,a)}+\frac{t}{\widetilde{R}^2(\theta,a)},
\end{split}
\end{equation*}
and this leads to an acceptable contribution using Lemma \ref{lem:ControlOnEField} and Lemma \ref{LemDensity}. We now turn to
\begin{equation*}
\begin{split}
\partial_\theta\partial_\theta\PPsi=-\frac{{\bf m}(\widetilde{R})}{\widetilde{R}^2}\left(\partial_\theta\partial_\theta\widetilde{R}-2\frac{(\partial_\theta\widetilde{R})^2}{\widetilde{R}}\right)-\bm\varrho(\widetilde{R})\frac{(\partial_\theta\widetilde{R})^2}{\widetilde{R}^2}.
\end{split}
\end{equation*}
Using Lemma \ref{DerRLem} and \eqref{EstimRa}, we see that
\begin{equation*}
\begin{split}
\vert\partial_\theta\partial_\theta\widetilde{R}\vert+\frac{\vert\partial_\theta\widetilde{R}\vert^2}{\widetilde{R}}\lesssim \widetilde{R}^{-1}(\theta,a),\qquad a^{-2}\vert\partial_\theta\partial_\theta\widetilde{R}\vert+a^{-2}\frac{\vert\partial_\theta\widetilde{R}\vert^2}{\widetilde{R}}\lesssim 1,
\end{split}
\end{equation*}
and this term can be handled as before using Lemma \ref{lem:ControlOnEField} and Lemma \ref{LemDensity}. Finally, we compute that
\begin{equation*}
\begin{split}
\partial_a\partial_a\PPsi=-\frac{{\bf m}(\widetilde{R})}{\widetilde{R}^2}\left(\partial_a\partial_a\widetilde{R}-2\frac{(\partial_a\widetilde{R})^2}{\widetilde{R}}\right)-\bm\varrho(\widetilde{R})\frac{(\partial_a\widetilde{R})^2}{\widetilde{R}^2}.
\end{split}
\end{equation*}
Using \eqref{BdRTFirstDer} and \eqref{BdRTSecondDer}, we find that
\begin{equation*}
\begin{split}
a^2\vert\partial_a\partial_a\widetilde{R}\vert+a^2\vert\frac{(\partial_a\widetilde{R})^2}{\widetilde{R}}\vert&\lesssim\frac{t^2}{\widetilde{R}^2}+\frac{t}{a\widetilde{R}}+\frac{\ln\langle \frac{a^2}{q}\widetilde{R}\rangle}{\frac{a^2}{q}\widetilde{R}}\widetilde{R}+\left(\frac{\ln\langle\frac{a^2}{q}\widetilde{R}\rangle}{\frac{a^2}{q}\widetilde{R}}\right)^2\widetilde{R}+a^2\frac{t^2}{\widetilde{R}}\\
&\lesssim (1+a^2)\left(1+\widetilde{R}(\theta,a)+t^2/\widetilde{R}\right)
\end{split}
\end{equation*}
and that
\begin{equation*}
\begin{split}
a^2\frac{(\partial_a\widetilde{R})^2}{\widetilde{R}^2}&\lesssim a^2\frac{t^2}{\widetilde{R}^2}+\left(\frac{\ln\langle\frac{a^2}{q}\widetilde{R}\rangle}{\frac{a^2}{q}\widetilde{R}}\right)^2\lesssim 1+\frac{a^2t^2}{\widetilde{R}^2}.
\end{split}
\end{equation*}
Using Lemma \ref{lem:ControlOnEField}, we find that
\begin{equation*}
\begin{split}
a^2\frac{{\bf m}(\widetilde{R})}{\widetilde{R}^2}\left\vert\partial_a\partial_a\widetilde{R}-2\frac{(\partial_a\widetilde{R})^2}{\widetilde{R}}\right\vert&\lesssim \frac{1+a^2}{t}\left(\Vert (a^{-\frac{1}{2}}+a^{-\frac{3}{2}})\gamma\Vert_{L^2_{\theta,a}}+t^{-\frac{1}{4}}\Vert (a^{-\frac{7}{2}}+\theta^{\frac{3}{2}}+\theta^{\frac{7}{2}})\gamma\Vert_{L^2_{\theta,a}}^2\right)
\end{split}
\end{equation*}
while using Lemma \ref{LemDensity}, we find that
\begin{equation*}
\begin{split}
a^2\bm\varrho(\widetilde{R})\frac{(\partial_a\widetilde{R})^2}{\widetilde{R}^2}
&\lesssim \left(1+\frac{a^2t^2}{\widetilde{R}^2}\right)\bm\varrho(\widetilde{R}(\theta,a))\\
&\lesssim \frac{1+a^2}{t}\left[\Vert a\partial_a\gamma\Vert_{L^2_{\theta,a}}\Vert (a^{-1}+a^{-3})\gamma\Vert_{L^2_{\theta,a}}+\Vert(a^{-\frac{1}{2}}+a^{-\frac{3}{2}})\gamma\Vert_{L^2_{\theta,a}}^2\right]
+\frac{1+a^2}{t^\frac{4}{3}}N_1.
\end{split}
\end{equation*}
This finishes the proof.
\end{proof}

\section{Nonlinear analysis}\label{SecNLin}

In this section we consider the full nonlinear equation \eqref{eq:VPPoisson},
\begin{equation}
 \partial_t\gamma=\lambda \{\PPsi,\gamma\}.
\end{equation}
We first establish global existence of strong solutions via a bootstrap in Section \ref{SSecBootstap}, then we demonstrate the modified scattering asymptotics in Section \ref{ModifiedScat}. This establishes Proposition \ref{prop:deriv_boot} and Theorem \ref{thm:mainfull}.

\subsection{Bootstraps and global existence}\label{SSecBootstap}

We first propagate global bounds using energy estimates. The key property we will use is that the integral of a Poisson bracket vanishes. Commuting with appropriate operators gives the equations
\begin{equation}\label{CommEq1}
\begin{split}
\partial_t(a^p\gamma)-\lambda\left\{\PPsi,a^p\gamma\right\}&=-\frac{p\lambda}{a}\{\PPsi,a\}a^p\gamma=p\lambda\partial_\theta\PPsi\cdot a^{p-1}\gamma,\\
\partial_t(\theta^p\gamma)-\lambda\left\{\PPsi,\theta^p\gamma\right\}&=-p\lambda\partial_a\PPsi\cdot \theta^{p-1}\gamma,\\
\partial_t(\partial_\beta\gamma)-\lambda\left\{\PPsi,\partial_\beta\gamma\right\}&=\lambda\{\partial_\beta\PPsi,\gamma\},\qquad\beta\in\{\theta,a\}.
\end{split}
\end{equation}

The key in the bootstrap estimates is that one can propagate $a$ moments easily and that the terms with slowest decay involve only these moments (see ${\bf m}_s$ in \eqref{DecMass} and {$\bm\varrho_s$} in \eqref{DecRho}). Interestingly, we will see in Section \ref{MomentSSSec} that the moments can be bootstrapped on their own, allowing global bounds on weak solutions. These moment bounds allow to propagate another bootstrap for higher regularity. For simplicity, we only propagate the first order derivatives in Section \ref{BootDerSSSec}.

\subsubsection{Moment Bootstrap}\label{MomentSSSec}

It turns out that control of the moments can be bootstrapped independently of any derivative bound.

\begin{lemma}\label{MomBdsLem}

Let $p\ge2$ and assume that $\gamma$ solves \eqref{eq:VPPoisson} on an interval $0\le t\le T$ and assume that
\begin{equation}\label{BootAssMom}
\begin{split}
\Vert\left(a^{-3p}+a^p+\theta^p\right)\gamma(t=0)\Vert_{L^2_{\theta,a}}&\le\varepsilon_0,\\
\Vert\left(a^{-3p}+a^p+\theta^p\right)\gamma(t)\Vert_{L^2_{\theta,a}}&\le\varepsilon_1\langle t\rangle^\delta
\end{split}
\end{equation}
then there holds that
\begin{equation*}
\begin{split}
\Vert \left(a^{-3p}+a^p\right)\gamma(t)\Vert_{L^2_{\theta,a}}&\lesssim\varepsilon_0,\qquad
\Vert \theta^p\gamma(t)\Vert_{L^2_{\theta,a}}\lesssim \varepsilon_0+\varepsilon_1\langle t\rangle^{\varepsilon_0}.
\end{split}
\end{equation*}

\end{lemma}

\begin{proof}[Proof of Lemma \ref{MomBdsLem}]

The moments can be readily estimated.
By \eqref{CommEq1} we have that, for $q\in\mathbb{R}$
\begin{equation*}
 \frac{1}{2}\frac{d}{dt}\norm{a^q\gamma}_{L^2_{\theta,a}}^2\lesssim\Vert a^{-1}\partial_\theta\PPsi\Vert_{L^\infty}\norm{a^q\gamma}_{L^2_{\theta,a}}^2.
\end{equation*}
Using \eqref{BdsDThetaPsi}, the bootstrap assumptions \eqref{BootAssMom} and Gronwall inequality, we find that
\begin{equation}\label{GlobalBdMoma}
\begin{split}
\Vert a^q\gamma(t)\Vert_{L^2_{\theta,a}}^2&\lesssim \Vert a^q\gamma(0)\Vert_{L^2_{\theta,a}}^2\lesssim\varepsilon_0^2.
\end{split}
\end{equation}
Similarly, for $q\ge0$, using \eqref{CommEq1} and \eqref{BdsDaPsi}, we find that
\begin{equation*}
\begin{split}
\frac{1}{2}\frac{d}{dt}\norm{\theta^q\gamma}_{L^2_{\theta,a}}^2&\lesssim\iint \vert\partial_a\PPsi(\theta,a)\vert\cdot \vert\theta\vert^q\gamma\cdot\vert\theta\vert^{q-1}\gamma \,dad\theta\\
&\lesssim t^{-1}\Vert (1+a^{-1})\gamma\Vert_{L^2_{\theta,a}}^2\cdot\left[\Vert \theta^q\gamma\Vert_{L^2_{\theta,a}}^2+\Vert a^{-3}\theta^{q-1}\gamma\Vert_{L^2_{\theta,a}}\Vert \theta^q\gamma\Vert_{L^2_{\theta,a}}\right]\\
&\quad+t^{-\frac{5}{4}}\Vert\theta^2\gamma\Vert_{L^2_{\theta,a}}\Vert \theta^q\gamma\Vert_{L^2_{\theta,a}}\Vert \theta^{q-1}\gamma\Vert_{L^2_{\theta,a}}
\end{split}
\end{equation*}
and we can again apply Gronwall estimate.
 \end{proof}

\subsubsection{Control on the derivatives}\label{BootDerSSSec}

We now show that we can obtain strong solutions by bootstrapping control of derivatives. It turns out that we will also need some moments of first derivatives. Given a weight function $\omega(\theta,a)$, we define
\begin{equation*}
\omega^{(1)}(\theta,a):=(a+a^{-1})\omega(\theta,a),\qquad\omega^{(2)}(\theta,a):=a\omega(\theta,a),
\end{equation*}
and we compute that
\begin{equation}\label{WeightEvolEq}
\begin{split}
\partial_t(\omega^{(1)}\gamma_\theta)-\lambda\{\PPsi,\omega^{(1)}\gamma_\theta\}-\lambda\partial_\theta\partial_a\PPsi\cdot\omega^{(1)}\gamma_\theta+\lambda\frac{a+a^{-1}}{a}\partial_\theta\partial_\theta\PPsi\cdot\omega^{(2)}\gamma_a&=\frac{\lambda}{a}\partial_\theta\PPsi \cdot a\partial_a\omega^{(1)}\gamma_\theta\\
&\quad -\lambda\partial_a\PPsi\cdot\partial_\theta\omega^{(1)}\gamma_\theta,\\
\partial_t(\omega^{(2)}\gamma_a)-\lambda\{\PPsi,\omega^{(2)}\gamma_a\}+\lambda\partial_\theta\partial_a\PPsi\cdot\omega^{(2)}\gamma_a-\lambda\frac{a^2}{1+a^2}\partial_a\partial_a\PPsi\cdot\omega^{(1)}\gamma_\theta&=\frac{\lambda}{a}\partial_\theta\PPsi \cdot a\partial_a\omega^{(2)}\gamma_a\\
&\quad-\lambda\partial_a\PPsi\cdot\partial_\theta\omega^{(2)}\gamma_a.\\
\end{split}
\end{equation}
We will need this when
\begin{equation}\label{CompatibleSet}
\omega\in\mathcal{I}_{0}:=\{1,\,\alpha^{-3},\,\alpha^{-6},\,\theta,\,\theta\alpha^{-3},\,\theta^2\}.
\end{equation}
More generally, one can consider $\omega_{p,q}:=\theta^pa^q$ for $p\in\mathbb{N}_{0}$, and $a\in\mathbb{R}$ and then the properties we need are that
\begin{equation}\label{WeightProp}
\vert\omega^{(1)}_{p,q}\vert\ge \vert\omega_{p,q}\vert,\qquad\vert a\partial_a\omega^{(j)}_{p,q}\vert\lesssim \vert\omega^{(j)}_{p,q}\vert,\qquad \vert\partial_\theta\omega^{(j)}_{p,q}\vert\lesssim p\omega^{(j)}_{p-1,q},
\end{equation}
and that the list of weights $\mathcal{I}=\{\omega_{p,q}\}_{p,q}$ that we consider satisfies the induction property
\begin{equation}\label{InductionProperty}
\begin{split}
\omega_{p,q}\in \mathcal{I}\,\,\Rightarrow \omega_{p-1,q-3},\omega_{p-1,q}\in\mathcal{I}.
\end{split}
\end{equation}
We call such sets of weights {\it compatible}.

\begin{proposition}\label{BootstrapDer}

Let $\mathcal{I}$ be a compatible list. Assume that $\gamma$ solves \eqref{eq:VPPoisson} for $0\le t\le T$ and satisfies for any weight
\begin{equation}\label{AssBoot1}
\begin{split}
\Vert \omega^{(1)}_{p,q}\partial_\theta\gamma(t=0)\Vert_{L^2_{\theta,a}}+\Vert \omega^{(2)}_{p,q}\partial_a\gamma(t=0)\Vert_{L^2_{\theta,a}}+\Vert \left(a^{20}+a^{-20}+\theta^{20}\right)\gamma(t=0)\Vert_{L^2_{\theta,a}}&\le\varepsilon_0,\\
\Vert \left(a^{20}+a^{-20}+\theta^{20}\right)\gamma(t)\Vert_{L^2_{\theta,a}}&\le \varepsilon_1\langle t\rangle^\delta,\\
\Vert \omega^{(1)}_{p,q}\partial_\theta\gamma(t)\Vert_{L^2_{\theta,a}}+\Vert \omega^{(2)}_{p,q}\partial_a\gamma(t)\Vert_{L^2_{\theta,a}}&\le \varepsilon_1\langle t\rangle^{(p+1)\delta}.
\end{split}
\end{equation}
Then the following stronger bounds hold for any weights
\begin{equation}\label{BootDer}
\begin{split}
\Vert \omega^{(1)}_{p,q}\partial_\theta\gamma\Vert_{L^2_{\theta,a}}&\lesssim \varepsilon_0+\varepsilon_1^\frac{3}{2}\langle t\rangle^{p\delta},\qquad
\Vert \omega^{(2)}_{p,q}\partial_a\gamma\Vert_{L^2_{\theta,a}}\lesssim\varepsilon_0+\varepsilon_1^\frac{3}{2}\langle t\rangle^{(p+1)\delta}.
\end{split}
\end{equation}

\end{proposition}

In particular, the case of $\omega=1$ gives the result of Proposition \ref{prop:deriv_boot}.

\begin{proof}[Proof of Proposition \ref{BootstrapDer}]
Writing $\omega=\omega_{p,q}$ for simplicity of notation, using \eqref{WeightEvolEq} we find that
\begin{equation*}\label{EstimGowthNormDer}
\begin{split}
\frac{d}{dt}\Vert \omega^{(1)}\gamma_\theta\Vert_{L^2_{\theta,a}}^2&\lesssim\iint\left[\vert\partial_\theta\partial_a\PPsi\vert+\vert\frac{1}{a}\partial_\theta\PPsi\vert\right]\cdot (\omega^{(1)}\gamma_\theta)^2\, d\theta da+\!\iint (1+a^{-2})\vert\partial_\theta\partial_\theta\PPsi\vert\cdot\vert \omega^{(2)}\gamma_a\cdot \omega^{(1)}\gamma_\theta\vert\, d\theta da\\
&\quad+\iint \vert\partial_a\PPsi\vert \cdot\vert \partial_\theta\omega^{(1)}\gamma_\theta\cdot \omega^{(1)}\gamma_\theta\vert\, d\theta da,
\end{split}
\end{equation*}
where we have used $\vert a\partial_a\omega^{(j)}\vert\lesssim \omega^{(j)}$ from \eqref{WeightProp}. The first two terms on each right hand side lead directly to a Gronwall bootstrap using \eqref{BdsDThetaPsi} and \eqref{ControlSecondDerPsiEst1}. The last is not present when $p=0$. If $p\ge1$, we may use the induction property \eqref{WeightProp} with \eqref{BdsDaPsi} to proceed as follows:
\begin{equation*}
\begin{split}
&\iint \vert\partial_a\PPsi\vert\cdot \vert \partial_\theta\omega^{(1)}_{p,q}\gamma_\theta\cdot \omega^{(1)}_{p,q}\gamma_\theta\vert d\theta da\\
\lesssim& t^{-1}\Vert (1+a^{-1})\gamma\Vert_{L^2_{\theta,a}}^2\Vert \omega^{(1)}_{p,q}\gamma_\theta\Vert_{L^2_{\theta,a}}\Vert\omega_{p-1,q}^{(1)}\gamma_\theta\Vert_{L^2_{\theta,a}}+t^{-1}\Vert\gamma\Vert_{L^2_{\theta,a}}^2\left[\Vert\omega^{(1)}_{p,q}\gamma_\theta\Vert_{L^2_{\theta,a}}^2+\Vert\omega^{(1)}_{p-1,q-3}\gamma\Vert_{L^2_{\theta,a}}^2\right]\\
&\quad+t^{-\frac{5}{4}}\Vert \theta^2\gamma\Vert_{L^2_{\theta,a}}^2\Vert \omega^{(1)}_{p,q}\gamma\Vert_{L^2_{\theta,a}}\Vert \omega^{(1)}_{p-1,q}\gamma\Vert_{L^2_{\theta,a}},
\end{split}
\end{equation*}
and we see that all terms lead to \eqref{BootDer}.

Similarly, we compute that
\begin{equation*}\label{EstimGowthNormDer2}
\begin{split}
\frac{d}{dt}\Vert \omega^{(2)}\gamma_a\Vert_{L^2_{\theta,a}}^2&\lesssim\iint\left[\vert\partial_\theta\partial_a\PPsi\vert+\vert\frac{1}{a}\partial_\theta\PPsi\vert\right]\cdot (\omega^{(2)}\gamma_a)^2 d\theta da+\iint \frac{a^2}{1+a^2}\vert\partial_a\partial_a\PPsi\vert\cdot\vert \omega^{(2)}\gamma_a\cdot \omega^{(1)}\gamma_\theta\vert d\theta da\\
&\quad+\iint \vert\partial_a\PPsi\vert\cdot \vert \partial_\theta\omega^{(2)}\gamma_a\cdot \omega^{(2)}\gamma_a\vert d\theta da.
\end{split}
\end{equation*}
Here the only new term is the second one on the right hand side. For this term, we use \eqref{ControlSecondDerPsiEst2} to get
\begin{equation*}
\begin{split}
\iint \frac{a^2}{1+a^2}\vert\partial_a\partial_a\PPsi\vert\cdot\vert \omega^{(2)}\gamma_a\cdot \omega^{(1)}\gamma_\theta\vert d\theta da&\lesssim t^{-1}\Vert a\gamma_a\Vert_{L^2_{\theta,a}}\Vert a^{-1}\gamma\Vert_{L^2_{\theta,a}}\Vert \omega^{(2)}\gamma_a\Vert_{L^2}\Vert \omega^{(1)}\gamma_\theta\Vert_{L^2_{\theta,a}}\\
&\quad+t^{-1}\Vert (a^2+a^{-2})\gamma\Vert_{L^2_{\theta,a}}^2\Vert \omega^{(2)}\gamma_a\Vert_{L^2}\Vert \omega^{(1)}\gamma_\theta\Vert_{L^2_{\theta,a}}\\
&\quad+t^{-\frac{5}{4}}N_1\cdot\Vert \omega^{(2)}\gamma_a\Vert_{L^2}\Vert \omega^{(1)}\gamma_\theta\Vert_{L^2_{\theta,a}}
\end{split}
\end{equation*}
and since we have already controlled the $\theta$ derivative, we may use \eqref{BootDer} to obtain that
\begin{equation*}
\begin{split}
\Vert a\gamma_a\Vert_{L^2_{\theta,a}}\Vert \omega^{(1)}_{p,q}\gamma_\theta\Vert_{L^2_{\theta,a}}&\lesssim \varepsilon_1^2\langle t\rangle^{(p+1)\delta},
\end{split}
\end{equation*}
which gives an acceptable contribution with Gronwall's estimate.
\end{proof}

\subsection{Asymptotic behavior}\label{ModifiedScat}

The analysis in this section is partially inspired by \cite{IPWW2020,Pan2020}.

\subsubsection{Weak-strong limit and asymptotic electric field}\label{ConvScatteringQties}

Before we obtain strong convergence of the particle distribution, we first need weak convergence of ``asymptotic functions'' which are defined in terms of averages along linearized trajectories. Given a bounded measurable function $\tau$, we define
\begin{equation*}
\begin{split}
\langle\tau\rangle(t)&:=\iint\tau(\alpha)\gamma^2(\vartheta,\alpha,t)d\vartheta d\alpha.
\end{split}
\end{equation*}
The following Lemma states that these averages converge.

\begin{lemma}\label{ScatteringAverages}
Assume that $\gamma$ solves \eqref{eq:VPPoisson} for $0\le t\le T$ and satisfies the conclusions of Proposition \ref{MomBdsLem} for $p=2$ and the conclusions of Proposition \ref{BootstrapDer}. Given any bounded function $\tau(a)$, the limit
\begin{equation*}
\begin{split}
\langle\tau\rangle_\infty:=\lim_{t\to\infty}\iint\tau(\alpha)\gamma^2(\vartheta,\alpha,t)d\vartheta d\alpha
\end{split}
\end{equation*}
exists and satisfies
\begin{equation}\label{ConvergenceAverages}
\vert\langle\tau\rangle(t)-\langle\tau\rangle_\infty\vert\lesssim  \varepsilon_1^4t^{-\frac{1}{4}}.
\end{equation}

\end{lemma}

\begin{proof}[Proof of Lemma \ref{ScatteringAverages}]

Using \eqref{eq:VPPoisson}, we see that
\begin{equation*}
\begin{split}
\left\vert\frac{d}{dt}\langle\tau\rangle(t)\right\vert&\lesssim\abs{\iint \tau(\alpha)\partial_\theta\PPsi\cdot\gamma(\vartheta,\alpha)\cdot\partial_a \gamma(\vartheta,\alpha)d\vartheta d\alpha}+\abs{\iint \tau(\alpha)\partial_\alpha\partial_\theta\PPsi\cdot\gamma(\vartheta,\alpha)\cdot\partial_\theta \gamma(\vartheta,\alpha)d\vartheta d\alpha}\\
&\lesssim \left[\Vert a^{-1}\partial_\theta\PPsi\Vert_{L^\infty_{\theta,a}}\Vert a\partial_a\gamma\Vert_{L^2_{\theta,a}} +\norm{\partial_\theta\partial_a\PPsi}_{L^\infty_{\theta,a}}\norm{\partial_\theta\gamma}_{L^2_{\theta,a}}\right]\Vert \tau(\alpha)\gamma(\vartheta,\alpha)\Vert_{L^2_{\theta,a}}.
\end{split}
\end{equation*}
Proposition \ref{BoundsOnE}, Proposition \ref{ControlSecondDerPsi} and \eqref{AssBoot1} then show that $\langle\tau\rangle(t)$ is a Cauchy sequence as $t\to\infty$. Integrating the time derivative then gives the bound \eqref{ConvergenceAverages}.

\end{proof}

The convergence of the scattering data allows to define the asymptotic electric potential and electric field
\begin{equation*}
\begin{split}
\Psi_\infty(a)&:=\lim_{t\to\infty}\iint\frac{1}{\max\{a,\alpha\}}\gamma^2(\vartheta,\alpha,t)d\vartheta d\alpha,\\
\mathcal{E}_\infty(a)&:=\frac{1}{a^2}\lim_{t\to\infty}\iint\mathfrak{1}_{\{\alpha\le a\}}\gamma^2(\vartheta,\alpha,t)d\vartheta d\alpha.
\end{split}
\end{equation*}
Informally, we expect that
\begin{equation*}
\begin{split}
\PPsi&=\frac{1}{t}\Psi_\infty(a)+o(t^{-1}),\qquad \widetilde{E}(\theta,a,t)=\frac{1}{t}\mathcal{E}_\infty(a)+o(t^{-1}).
\end{split}
\end{equation*}
Under our assumptions we can prove the following:
\begin{lemma}\label{AsymptoticEF}
Under the assumptions of Lemma \ref{ScatteringAverages}, there holds that
\begin{equation*}
\begin{split}
\vert\mathcal{E}_\infty(a)\vert\lesssim \varepsilon_1^2,\qquad \mathfrak{1}_{\mathcal{B}_\ast}\cdot\left\vert\partial_a\PPsi(\theta,a,t)+\frac{1}{t}\mathcal{E}_\infty(a)\right\vert&\lesssim t^{-\frac{6}{5}}\varepsilon_1^2,
\end{split}
\end{equation*}
where $\mathcal{B}_\ast$ is a smaller version of the bulk
\begin{equation*}
\begin{split}
\mathcal{B}_\ast:=\{(\theta,a)\,\, \vert\theta\vert\le t^\frac{1}{4},\,\, t^{-\frac{1}{4}}\le a\le t^\frac{1}{4}\}\subset\mathcal{B}.
\end{split}
\end{equation*}

\end{lemma}

\begin{proof}[Proof of Lemma \ref{AsymptoticEF}]

The first estimate follows from the uniform bound
\begin{equation*}
\begin{split}
\frac{1}{a^2}\left\vert \iint\mathfrak{1}_{\{\alpha\le a\}}\gamma^2(\vartheta,\alpha,t)d\vartheta d\alpha\right\vert&\le \iint\mathfrak{1}_{\{\alpha\le a\}}\alpha^{-2}\gamma^2(\vartheta,\alpha,t)d\vartheta d\alpha\le\Vert a^{-1}\gamma\Vert_{L^2_{\theta,a}}^2\lesssim\varepsilon_1^2.
\end{split}
\end{equation*}
We now turn to the second estimate. Recall from the proof of Proposition \ref{BoundsOnE} that
\begin{equation*}
\begin{split}
\partial_a\PPsi(\theta,a,t)&=-\frac{{\bf m}(\widetilde{R}(\theta,a),t)}{\widetilde{R}^2(\theta,a)}\left[t\partial_\theta R(\theta+ta,a)+\partial_aR(\theta+ta,a)\right]=-t\frac{{\bf m}(\widetilde{R}(\theta,a),t)}{\widetilde{R}^2(\theta,a)}+\mathcal{R}_1,\\
\mathfrak{1}_{\mathcal{B}}\vert\mathcal{R}_1\vert &\le 
\frac{{\bf m}(\widetilde{R}(\theta,a),t)}{\widetilde{R}^2(\theta,a)}\left(\frac{tq}{a^2\widetilde{R}(\theta,a)}+\frac{q}{a^3}\ln\langle\frac{a^2}{q}\widetilde{R}(\theta,a)\rangle\right)\lesssim \varepsilon_1^2t^{-\frac{6}{5}},
\end{split}
\end{equation*}
where we have used  Lemma \ref{DerRLem}. Furthermore, with $\mathcal{E}(a,t):=a^{-2}\iint\mathfrak{1}_{\alpha\leq a}\gamma^2(\vartheta,\alpha,t)$, we have
\begin{equation*}
\begin{split}
t\frac{{\bf m}(at,t)}{(at)^2}&=\frac{1}{a^2t}\iint\mathfrak{1}_{\{R(\vartheta+t\alpha,\alpha)\le at\}}\gamma^2(\vartheta,\alpha,t)\, d\vartheta d\alpha=\frac{1}{t}\left[\mathcal{E}(a,t)+\mathcal{R}_2\right],\\
\vert \mathcal{R}_2\vert&\le \frac{1}{a^2}\iint\mathfrak{1}_{\mathcal{S}_1\cup\mathcal{S}_2}\gamma^2(\vartheta,\alpha,t)\, d\vartheta d\alpha,
\end{split}
\end{equation*}
where
\begin{equation*}
\begin{split}
\mathcal{S}_1\cup\mathcal{S}_2&=\{R(\vartheta+t\alpha,\alpha)\le at\}\triangle\{\alpha t\le at\},\\
\mathcal{S}_1&:=\{\alpha\leq a,\,\, R(\vartheta+t\alpha,\alpha)\ge at\},\qquad
\mathcal{S}_2:=\{\alpha\ge a,\,\, R(\vartheta+t\alpha,\alpha)\le at\}.
\end{split}
\end{equation*}
Note from \eqref{EstimG}, there holds that
\begin{equation*}
\begin{split}
\vert\vartheta+t\alpha\vert\ge \widetilde{R}(\vartheta,\alpha)\ge \vert\vartheta+t\alpha\vert-\frac{q}{a^2}\ln\langle\frac{a^2\vert\vartheta+t\alpha\vert}{q}\rangle,
\end{split}
\end{equation*}
so that on the support of $\mathcal{B}_\ast$
\begin{equation}\label{EstimRBast}
\vert \widetilde{R}(\vartheta,\alpha)-t\alpha\vert\le Ct^{\frac{1}{4}}
\end{equation}
for some universal constant $C>0$. Therefore we have
\begin{equation*}
\begin{split}
\mathcal{B}_\ast\cap\{\mathcal{S}_1\cup\mathcal{S}_2\}\subset\{a-Ct^{-\frac{1}{4}}\le \alpha\le a+Ct^{-\frac{1}{4}}\},
\end{split}
\end{equation*}
so that, using \eqref{AssBoot1},
\begin{equation*}
\begin{split}
\mathfrak{1}_{\mathcal{B}_\ast}\cdot\vert\mathcal{R}_2\vert &\lesssim \frac{1}{a^2}\mathfrak{1}_{\mathcal{B}_\ast}\cdot\iint\mathfrak{1}_{\mathcal{B}^c_\ast}\gamma^2(\vartheta,\alpha)\, d\vartheta d\alpha+\frac{1}{a^2}\iint\mathfrak{1}_{\{\vert \alpha-a\vert\le Ct^{-\frac{1}{4}}\}}\gamma^2(\vartheta,\alpha)\, d\vartheta d\alpha\\
&\lesssim t^{\frac{1}{2}}\iint\mathfrak{1}_{\mathcal{B}^c_\ast}\gamma^2(\vartheta,\alpha)\, d\vartheta d\alpha+\iint\mathfrak{1}_{\{\vert \alpha-a\vert\le Ct^{-\frac{1}{4}}\}}\left(\frac{\alpha}{a}\right)^2\alpha^{-2}\gamma^2(\vartheta,\alpha)\, d\vartheta d\alpha\\
&\lesssim t^{-\frac{1}{5}}\varepsilon_1^2.
\end{split}
\end{equation*}
Finally, using \eqref{EstimRBast} with Lemma \ref{lem:ControlOnEField} and Lemma \ref{LemDensity},
\begin{equation*}
\begin{split}
\mathfrak{1}_{\mathcal{B}_\ast}\left\vert \frac{{\bf m}(\widetilde{R}(\theta,a))}{\widetilde{R}^2(\theta,a)}-\frac{{\bf m}(at)}{a^2t^2}\right\vert& \le \mathfrak{1}_{\mathcal{B}_\ast}\cdot\frac{\vert{\bf m}(\widetilde{R}(\theta,a))-{\bf m}(at)\vert}{a^2t^2}+\mathfrak{1}_{\mathcal{B}_\ast}\cdot\frac{{\bf m}(\widetilde{R}(\theta,a))}{\widetilde{R}^2(\theta,a)}\left\vert 1-\frac{\widetilde{R}^2(\theta,a)}{a^2t^2}\right\vert\\
&\lesssim \mathfrak{1}_{\mathcal{B}_\ast}\cdot\frac{\vert \widetilde{R}(\theta,a)-at\vert}{(at)^2}\cdot \sup_{r} \bm\varrho(r,t)+\mathfrak{1}_{\mathcal{B}_\ast}\cdot\varepsilon_1^2t^{-2}\frac{\vert \widetilde{R}(\theta,a)-at\vert}{at}\\
&\lesssim\varepsilon_1^2 t^{-\frac{6}{5}}.
\end{split}
\end{equation*}
Since by \eqref{ConvergenceAverages} we have that $\abs{\mathcal{E}(a,t)-\mathcal{E}_\infty(a)}\lesssim \varepsilon_1^4 t^{-\frac{1}{4}}$, this concludes the proof.
\end{proof}

\subsubsection{Strong limit}\label{ConvPDF}

We can now correct the trajectories to get a strong limit and prove our main theorem.

\begin{proof}[Proof of Theorem \ref{thm:mainfull}]
Under these conditions, we may apply Lemma \ref{MomBdsLem} and Proposition \ref{BootstrapDer} to propagate global bounds on the moments and derivatives with the weights in \eqref{CompatibleSet}. Lemma \ref{ScatteringAverages} justifies the existence of $\mathcal{E}_\infty$ and we have the estimates in Lemma \ref{AsymptoticEF}. Let
\begin{equation*}
\begin{split}
\sigma(\theta,a,t)&:=\gamma(\theta+\lambda\ln t\cdot\mathcal{E}_\infty(a),a,t).
\end{split}
\end{equation*}
We claim that $\sigma$ converges to a limit $\sigma_\infty$ in $L^2_{\theta,a}$. Indeed we compute that
\begin{equation*}
\begin{split}
\partial_t\sigma(\theta,a,t)&=\lambda\left(\partial_a\PPsi(\theta^\ast,a,t)+\frac{1}{t}\mathcal{E}_\infty(a)\right)\partial_\theta\gamma(\theta^\ast,a,t)-\lambda\partial_\theta\PPsi(\theta^\ast,a,t)\partial_a\gamma(\theta^\ast,a,t),\\
\theta^\ast&=\theta+\lambda \ln t\cdot\mathcal{E}_\infty(a).
\end{split}
\end{equation*}
We directly obtain that
\begin{equation*}
\begin{split}
\Vert \partial_\theta\PPsi(\theta^\ast,a,t)\partial_a\gamma(\theta^\ast,a,t)\Vert_{L^2_{\theta,a}}&\lesssim \Vert a^{-1}\partial_\theta\PPsi\Vert_{L^\infty_{\theta,a}}\Vert a\partial_a\gamma\Vert_{L^2_{\theta,a}}\lesssim \varepsilon_1^2t^{-\frac{5}{4}},
\end{split}
\end{equation*}
while using Lemma \ref{AsymptoticEF}, we find that
\begin{equation*}
\begin{split}
\Vert\mathfrak{1}_{\mathcal{B}_\ast}\left(\partial_a\PPsi(\theta^\ast,a,t)+\frac{1}{t}\mathcal{E}_\infty(a)\right)\partial_\theta\gamma(\theta^\ast,a,t)\Vert_{L^2_{\theta,a}}&\lesssim \varepsilon_1^2 t^{-\frac{6}{5}}\Vert\partial_\theta\gamma\Vert_{L^2_{\theta,a}}\lesssim \varepsilon_1^3 t^{-\frac{9}{8}}.
\end{split}
\end{equation*}
Using \eqref{AssBoot1} and \eqref{BdsDaPsi} yields
\begin{equation*}
\begin{split}
\Vert \mathfrak{1}_{\mathcal{B}^c_\ast}\frac{1}{t}\mathcal{E}_\infty(a)\partial_\theta\gamma(\theta^\ast,a,t)\Vert_{L^2_{\theta,a}}&\lesssim t^{-\frac{6}{5}}\varepsilon_1^2\Vert (\vert\theta\vert+a+a^{-1})\partial_\theta\gamma\Vert_{L^2_{\theta,a}}\lesssim \varepsilon_1^3 t^{-\frac{9}{8}}, \\
\Vert \mathfrak{1}_{\mathcal{B}^c_\ast}\partial_a\PPsi(\theta^\ast,a,t)\partial_\theta\gamma(\theta^\ast,a,t)\Vert_{L^2_{\theta,a}}&\lesssim\varepsilon_1^2 t^{-1}\Vert \mathfrak{1}_{\mathcal{B}_\ast}(1+\vert\theta\vert+a^{-3})\partial_\theta\gamma\Vert_{L^2_{\theta,a}}\lesssim \varepsilon_1^3 t^{-\frac{9}{8}}.
\end{split}
\end{equation*}
This establishes \eqref{L2Convergence}. In addition, the bounds from Proposition \ref{BootstrapDer} give uniform bounds on $\partial_\theta\gamma$ in $L^2_{\theta,a}$, which carries over to $\gamma_\infty$. Finally \eqref{SecondDefEF} follows from $L^2_{\theta,sa}$ convergence. Finally, the uniqueness of solutions follows by a simple Gronwall estimate on the $L^2_{\theta,a}$-norm of the difference of two solutions.
\end{proof}

\subsection*{Acknowledgments}
The authors would like to thank Y.\ Guo and P.\ Flynn for interesting and stimulating discussions. 

B.\ P.\ was supported in part by NSF grant DMS-1700282.

\bibliographystyle{abbrv}
\bibliography{vp-lib}

\end{document}